\documentclass{article}
\usepackage[utf8]{inputenc}
\usepackage[utf8]{inputenc}
\usepackage{amsmath}
\usepackage[left=3.4cm, right=3.4cm, bottom=3cm, top=3cm]{geometry}
\usepackage{hyperref}
\usepackage{cleveref}
\usepackage{amssymb}
\usepackage{mathtools}
\usepackage{tikz-cd}
\usepackage{amsthm}
\usepackage{graphicx}
\usepackage{graphics}
\usepackage{mathrsfs}
\usepackage{cleveref}
\usepackage{thmtools}
\usepackage{enumitem}
\newlist{steps}{enumerate}{1}
\setlist[steps, 1]{label = Step \arabic*:}
\theoremstyle{plain}
\newtheorem{theorem}{Theorem}[subsection]
\newtheorem{definition}[theorem]{Definition}

\newtheorem{proposition}[theorem]{Proposition}
\newtheorem{lemma}[theorem]{Lemma}
\newtheorem{remark}[theorem]{Remark}
\newtheorem{corollary}[theorem]{Corollary}
\theoremstyle{definition}
\newtheorem{exmp}[theorem]{Example}
\newtheorem{rmk}[theorem]{Remark}

\numberwithin{equation}{subsection}
\usepackage{rotating}
\usepackage{comment}
\usepackage[nottoc,notlot,notlof]{tocbibind} 

\title{Cohomology of Lie Groupoid Modules and the Generalized van Est Map}
\author{Joshua Lackman\thanks{University of Toronto; \href{mailto:jlackman@math.toronto.edu}{jlackman@math.toronto.edu}}}
\date{}

\begin{document}
\maketitle
\begin{abstract}
The van Est map is a map from Lie groupoid cohomology (with respect to a sheaf taking values in a representation) to Lie algebroid cohomology. We generalize the van Est map to allow for more general sheaves, namely to sheaves of sections taking values in a (smooth or holomorphic) $G$-module, where $G$-modules are structures which differentiate to representations. Many geometric structures involving Lie groupoids and stacks are classified by the cohomology of sheaves taking values in $G$-modules and not in representations, including $S^1$-groupoid extensions and equivariant gerbes. Examples of such sheaves are $\mathcal{O}^*$ and $\mathcal{O}^*(*D)\,,$ where the latter is the sheaf of invertible meromorphic functions with poles along a divisor $D\,.$ We show that there is an infinitesimal description of $G$-modules and a corresponding Lie algebroid cohomology. We then define a generalized van Est map relating these Lie groupoid and Lie algebroid cohomologies, and study its kernel and image. Applications include the integration of several infinitesimal geometric structures, including Lie algebroid extensions, Lie algebroid actions on gerbes, and certain Lie $\infty$-algebroids. 
\end{abstract}
\tableofcontents
\section{Introduction}
Many geometric structures involving Lie groupoids and stacks are classified by sheaf cohomology with respect to sheaves more general than sheaves of sections taking values in a (smooth) representation; often what is needed are sheaves of sections taking values in (smooth or holomorphic) $G$-modules. These include rank one representations, $S^1$-groupoid extensions and equivariant gerbes. In particular, these structures are related to flat connections on line bundles and on gerbes, to group actions on principal bundles, to the prequantization of symplectic groupoids and higher representations of Lie groups, and to the basic gerbe on a compact simple Lie group. All of these examples have known infinitesimal descriptions, and the generalized van Est map will show how to obtain these infinitesimal descriptions from the global geometric structure. This paper generalizes work done by Brylinski, Crainic, Weinstein and Xu in~\cite{brylinski},~\cite{Crainic} and~\cite{weinstein}, respectively.
\\\\The van Est map, as defined in~\cite{Crainic}, is a ``differentiation'' map; its domain is the cohomology of a Lie groupoid $G$ with coefficients taking values in a (smooth) representation, and its codomain is the cohomology of a Lie algebroid $\mathfrak{g}$ with coefficients taking values in a (smooth) $\mathfrak{g}$-representation. The generalized van Est map will have as its domain the cohomology of a Lie groupoid with coefficients taking values in a (smooth or holomorphic) $G$-module, which are generalizations of representations, and were defined by Tu in~\cite{Tu}; the codomain of this generalized van Est map will be 
the cohomology of a generalized Deligne complex, which is a generalization of the Chevalley-Eilenberg complex. \\\\The main theorem we prove determines the kernel and image of the van Est map up to a certain degree, depending on the connectivity of the source fibers of the groupoid. Applications include effective criteria for the integration of several infinitesimal geometric structures, including rank one representations, Lie algebroid extensions, Lie algebroid actions on gerbes, and certain Lie $\infty$-algebroids.
\\\\The results we prove help to compute not only the cohomology of sheaves on Lie groupoids, but also sheaves on stacks, as defined by Behrend and Xu in~\cite{Kai}. This is because a Lie groupoid defines a stack and a $G$-module defines a sheaf on a stack, and the cohomology of these two sheaves are isomorphic.
\subsection*{Outline of Paper}
The paper is organized as follows: section $0$ is a brief review of simplicial manifolds, Lie groupoids and stacks, but it is important for setting up notation, results and constructions which will be used in the next sections. This section contains all of the results about stacks that are needed for this paper. Section $1$ contains a review and a generalization of the Chevalley-Eilenberg complex. Section $2$ is where we define the van Est map and prove the main theorem of the paper. The next sections of the paper concern applications of the main theorem, various new constructions of geometric structures involving Lie groupoids, and examples.
\\\\$\mathbf{Notation:}$ For the rest of the paper, we use the following notation: given a smooth (or holomorphic) surjective submersion $\pi:M\to X\,,$ we let $\mathcal{O}(M)$ denote the sheaf of smooth (or holomorphic) sections of $\pi\,.$
\subsection*{Acknowledgements}
The author would like to thank his advisor Marco Gualtieri for introducing him to the problem and for many helpful discussions about the topic, as well as for helping him edit a preliminary draft of this paper. In addition, the author would also like to thank Francis Bischoff for many fruitful discussions and for suggesting edits to the preliminary draft. The author was supported by the University of Toronto's Faculty of Arts and Science Top Doctoral Award.

\begin{comment}
I would like to thank my advisor, Marco Gualtieri, for suggesting this problem to me and for many useful discussions, as well as for reading a preliminary version of this paper and providing me with feedback. I would also like to thank Francis Bischoff for many helpful discussions and for reading a preliminary version of this paper. I was supported by the University of Toronto and the faculty of arts and science top graduate award.
\end{comment}
\subsection{Simplicial Manifolds}
In this section we briefly review simplicial manifolds, sheaves on simplicial manifolds and their cohomology. 
\begin{comment}\begin{definition}[see \cite{Deligne}, \cite{Peters}]Let $Z^\bullet$ be a (semi) simplicial manifold, ie. a contravariant functor from the (semi) simplex category to the category of smooth manifolds. A sheaf on $Z^\bullet$ is a functor from the (semi) simplex category to the category of ordered pairs, whose objects are ordered pairs $(X,\mathcal{S})\,,$ where $X$ is a smooth manifold and $\mathcal{S}$ is a sheaf on $X\,,$ and whose morphisms are ordered pairs $(f,F_{\#}):(X,\mathcal{S})\to (Y,\mathcal{G})$ such that $f:X\to Y$ is a smooth map and $F_{\#}:\mathcal{G}\to f_*\mathcal{S}$ is a morphism of sheaves. A morphism between sheaves sheaves on a (semi) simplicial manifold is a natural transformation between the respective functors. 
$\blacksquare$\end{definition}
\end{comment}
%Concretely,
\begin{definition}Let $Z^\bullet$ be a (semi) simplicial manifold, ie. a contravariant functor from the (semi) simplex category to the category of manifolds. A sheaf $\mathcal{S}_\bullet$ on $Z^\bullet$ is a sheaf $\mathcal{S}_n$ on $Z^n\,,$ for all $n\ge 0\,,$ such that for each morphism $f:[n]\to[m]$ we have a morphism $\mathcal{S}(f):Z(f)^{-1}\mathcal{S}_n\to  \mathcal{S}_m\,,$ and such that $\mathcal{S}(f\circ g)=\mathcal{S}(f)\circ \mathcal{S}(g)\,.$ A morphism between sheaves $\mathcal{S}_\bullet$ and $\mathcal{G}_\bullet$ on $Z^\bullet$ is a morphism of sheaves $u^n:\mathcal{S}_n\to \mathcal{G}_n$ for each $n\ge 0$ such that for $f:[n]\to [m]$ we have that
$u^m\circ \mathcal{S}(f)=\mathcal{G}(f)\circ u^n\,.$  We let $\textrm{Sh}(Z^\bullet)$ denote the category of sheaves on $Z^\bullet\,.$ 
$\blacksquare$\end{definition}
\begin{definition}
Given a sheaf $\mathcal{S}_\bullet$ on a (semi) simplicial manifold $Z^\bullet\,,$ we define $Z^n:=\text{Ker}\big[\Gamma(\mathcal{S}_n)\xrightarrow{\delta^*} \Gamma(\mathcal{S}_{n+1})\big]\,,\,B^n:=\text{Im}\big[\Gamma(\mathcal{S}_{n-1})\xrightarrow{\delta^*}\Gamma(\mathcal{S}_n)\big]\,,$
where $\delta^*$ is the alternating sum of the face maps, ie.
\begin{align*}
    \delta^*=\sum_{i=0}^n(-1)^id_{n,i}^{-1}\,,
    \end{align*}
 where $d_{n,i}:Z^n\to Z^{n-1}$ is the $i^{\textrm{th}}$ face map. We then define the naive cohomology (see \cite{nlab}) \begin{equation*}H^n_{\text{naive}}(Z^\bullet\,,\mathcal{S}_\bullet):=Z^n/B^n\,.
\end{equation*}
$\blacksquare$\end{definition}
\begin{definition}[see~\cite{Deligne}]\label{cohomology simplicial} Given a (semi) simplicial manifold $Z^\bullet\,,$ $\textrm{Sh}(Z^\bullet)$ has enough injectives, and we define \begin{equation*}
H^n({Z^\bullet\,,\mathcal{S}_\bullet}):=R^n\Gamma_\text{inv}(\mathcal{S}_\bullet)\,,
\end{equation*}
where $\Gamma_\text{inv}:\textrm{Sh}(Z^\bullet)\to \mathbf{Ab}$ is given by $\mathcal{S}_\bullet \mapsto\text{Ker}\big[\Gamma(\mathcal{S}_0)\xrightarrow{\delta^*} \Gamma(\mathcal{S}_1)\big]\,.$
$\blacksquare$\end{definition}
\theoremstyle{definition}\begin{rmk}\label{acyclic resolution}
As usual, in addition to injective resolutions one can use acylic resolutions to compute cohomology.
\end{rmk}
\theoremstyle{definition}\begin{rmk}[see \cite{Deligne}]\label{double complex}
A convenient way to compute $H^*({Z^\bullet\,,\mathcal{S}_\bullet})$ is to choose a resolution 
\begin{align*}
    0\xrightarrow{}\mathcal{S}_\bullet\xrightarrow{}\mathcal{A}^0_\bullet\xrightarrow{\partial_\bullet^0}\mathcal{A}^1_\bullet\xrightarrow{\partial_\bullet^1}\cdots
\end{align*}
such that 
\begin{align*}
    0\xrightarrow{}\mathcal{S}_n\xrightarrow{}\mathcal{A}^0_n\xrightarrow{\partial_n^0}\mathcal{A}^1_n\xrightarrow{\partial_n^1}\cdots
\end{align*}
is an acyclic resolution of $\mathcal{S}_n\,,$ for all $n\ge 0\,,$ and then take the cohomology of the total complex of the double complex $C^p_q=\Gamma(A^p_q)\,,$  with differentials $\delta^*$ and $\partial^p_q\,.$
\end{rmk}
$\mathbf{}$
\\The following theorem is a well-known consequence of the Grothendieck spectral sequence: 
\begin{theorem}[Leray Spectral Sequence]\label{spectral}
Let $f:X^\bullet\to Y^\bullet$ be a morphism of simplicial topological spaces, and let $\mathcal{S}_\bullet$ be a sheaf on $X^\bullet\,.$ Then there is a spectral sequence $E^{pq}_*\,,$ called the Leray spectral sequence, such that $E^{pq}_2=H^p(Y^\bullet,R^qf_*(\mathcal{S}_\bullet))$ and such that 
\begin{align*}
 E^{pq}_2\Rightarrow H^{p+q}(X^\bullet,\mathcal{S}_\bullet)\,.
\end{align*}
\end{theorem}
\subsection{Stacks}
Here we briefly review the theory of differentiable stacks. A differentiable stack is in particular a category, and first we will define the objects of the category, and then the morphisms. All manifolds and maps can be taken to be in the smooth or holomorphic categories. The following definitions can be found in~\cite{Kai}.
\begin{definition}
Let $G\rightrightarrows G^0$ be a Lie groupoid. A $G$-principal bundle (or principal $G$-bundle) is a manifold $P$ together with a surjective submersion $P\overset{\pi}{\to}M$ and a map $P\overset{\rho}{\to} G^0\,,$ called the moment map, such that there is a right $G$-action on $P\,,$ ie.
a map $P\sideset{_{\rho}}{_{t}}{\mathop{\times}} G\to P\,,$ denoted $(p,g)\mapsto p\cdot g\,,$ such that
\begin{itemize}
    \item $\pi(p\cdot g)=\pi(p)$
    \item $\rho(p\cdot g)=s(g)$
    \item $(p\cdot g_1)\cdot g_2=p\cdot(g_1g_2)$
\end{itemize}
and such that 
\begin{align*}
    P\sideset{_{\rho}}{_{t}}{\mathop{\times}} G\to P\sideset{_{\pi}}{_{\pi}}{\mathop{\times}} P\,,\;(p,g)\mapsto (p,p\cdot g)
\end{align*}
is a diffeomorphism.
$\blacksquare$\end{definition}
\begin{definition}
A morphism between $G$-principal bundles $P\to M$ and $Q\to N$ is given by a commutative diagram of smooth maps
\[
\begin{tikzcd}
P\arrow{r}{\phi}\arrow{d} & Q\arrow{d}
\\M\arrow{r} & N
\end{tikzcd}
\]
such that $\phi(p\cdot g)=\phi(p)\cdot g\,.$ In particular this implies that $\rho\circ\phi(p)=\rho(p)\,.$
$\blacksquare$\end{definition}
\begin{definition}
Let $G\rightrightarrows G^0$ be a Lie groupoid. Then we define $[G^0/G]$ to be the category of $G$-principal bundles, together with its natural functor to the category of manifolds (which takes a $G$-principal bundle to its base manifold). We call $[G^0/G]$
a (differentiable or holomorphic) stack.
$\blacksquare$\end{definition}
%\begin{lemma}\label{stack groupoids eq}
%Groupoids $G\rightrightarrows G^0$ and %$K\rightrightarrows K^0$ are Morita equivalent if and %only if $[G^0/G]$ is categorically equivalent to %$[K^0/K]\,.$ 
%\end{lemma}
\theoremstyle{definition}\begin{rmk}
Given a Lie groupoid $G\rightrightarrows G^0\,,$ there is a canonical Grothendieck topology on $[G^0/G]\,,$ hence we can talk about sheaves on stacks and their cohomology. What is most important to know for the next sections is that a sheaf on $[G^0/G]$ is in particular a contravariant functor 
\[F:[G^0/G]\to \mathbf{Ab}\,.
\]
See Section~\ref{appendix:Cohomology of Sheaves on Stacks} for details.
\end{rmk}
\subsection{Groupoid Modules}
We now define Lie groupoid modules. Their importance is due to the fact that these are the structures which differentiate to representations; they will be one of the main objects we study in this paper.
\begin{definition}
Let $X$ be a manifold. A family of groups over $X$ is a Lie groupoid $M\rightrightarrows X$ such that the source and target maps are equal. A family of groups will be called a family of abelian groups if the multiplication on $M$ induces the structure of an abelian group on its source fibers, ie. if
$s(a)=s(b)$ then $m(a,b)=m(b,a)\,.$
$\blacksquare$\end{definition}
\theoremstyle{definition}\begin{exmp}\label{trivial}
Let $A$ be an abelian Lie group and let $X$ be a manifold. Then $X{\mathop{\times}} A$ is naturally a family of abelian groups, with the source
and target maps being the projection onto the first factor $p_1:X{\mathop{\times}} A\to X\,.$ This will be called a trivial family of abelian groups, and will be denoted $A_X\,.$
\end{exmp}
\theoremstyle{definition}\begin{exmp}\label{foag}
One way of constructing families of abelian groups is as follows: Let $A$ be an abelian group, and let $\text{Aut}(A)$ be its 
automorphism group. Then to any principal $\text{Aut}(A)$-bundle $P$ we have a canonical family of abelian groups - it is given by the fiber bundle
whose fibers are $A$ and which is associated to $P\,.$ Families of abelian groups constructed in this way are locally trivial in the sense
that locally they are isomorphic to the trivial family of abelian groups given by $A_{\mathbb{R}^n}\,,$ for some $n$ (compare this with vector bundles).
\end{exmp}
\begin{definition}(see \cite{Tu}$\,$): Let $G\rightrightarrows G^0$ be a Lie groupoid. A $G$-module $M$ is a family of abelian groups together with an action of $G$ on $M$ such that for $a\,,b\in G^0\,,$ $G(a,b):M_a\to M_b$ acts by homomorphisms $($here $G(a,b)$ is the set
of morphisms with source $a$ and target $b)\,.$ If $M$ is a vector bundle\footnote{Here we are implicitly using the fact that the forgeful functor from the category of finite dimensional vector spaces to the category of simply connected abelian Lie groups
is an equivalence of categories.}, $M$ will be called a representation of $G\,.$
$\blacksquare$\end{definition}
\theoremstyle{definition}\begin{exmp}
Let $G\rightrightarrows G^0$ be a groupoid and let $A$ be an abelian group. Then $A_{G^0}$ is a family of abelian groups (see Example~\ref{trivial}), and it is a $G$-module with the trivial action, that is $g\in G(x,y)$ acts by $g\cdot (x,a)=(y,a)\,.$ We will call this a trivial
$G$-module.
\end{exmp}
\begin{comment}\theoremstyle{definition}\begin{exmp}
The automorphism group of $S^1$ is $\mathbb{Z}_2\,,$ hence $H^1(X,\mathbb{Z}_2)$ classifies locally trivial families of groups over $X$ for which the fiber is isomorphic to $S^1\,,$ and all of these families can be equipped with a module structure for $\Pi_1(X)$ since the transition functions are locally constant. 
\end{exmp}
\end{comment}
\begin{comment}\theoremstyle{definition}\begin{exmp}
The automorphism group of $\mathbb{C}^*{\mathop{\times}}\mathbb{Z}$ is $\mathbb{C}^*{\mathop{\times}}\mathbb{Z}_2{\mathop{\times}}\mathbb{Z}_2\,,$ hence $H^1(X,\mathcal{O}^*{\mathop{\times}}\mathbb{Z}_2{\mathop{\times}}\mathbb{Z}_2)=H^1(X,\mathcal{O}^*){\mathop{\times}} H^1(X,\mathbb{Z}_2){\mathop{\times}} H^1(X,\mathbb{Z}_2)$ classifies locally trivial families of groups isomorphic to $\mathbb{C}^*{\mathop{\times}}\mathbb{Z}\,,$ over $X\,.$  In particular, any class in $H^1(X,\mathbb{C}^*){\mathop{\times}} H^1(X,\mathbb{Z}_2){\mathop{\times}} H^1(X,\mathbb{Z}_2)$ defines a $\Pi_1(X)$-module where the family of abelian groups is a locally trivial family of groups isomorphic to $\mathbb{C}^*{\mathop{\times}}\mathbb{Z}\,.$
\end{exmp}
\end{comment}
\theoremstyle{definition}\begin{exmp}\label{SL2}
Let $G=SL(2,\mathbb{Z})\,,$ which is the mapping class group of the torus. Every $T^2$-bundle over $S^1$ is isomorphic to one with transition functions in $SL(2,\mathbb{Z})\,,$ with the standard open cover of $S^1$ using two open sets. All of these are naturally $\Pi_1(S^1)$-modules since $SL(2,\mathbb{Z})$ is discrete. In particular, the Heisenberg manifold is a $\Pi_1(S^1)$-module. Explicitly,
consider the matrix 
\begin{align*}
    \begin{pmatrix}
    1 & 1 \\
    0 & 1 
    \end{pmatrix}\in SL(2,\mathbb{Z})\,.
\end{align*}
This matrix defines a map from $T^2\to T^2\,,$ and it corresponds to a Dehn twist. The total space of the corresponding $T^2$-bundle is diffeomorphic to the Heisenberg manifold $H_M$, which is the quotient of the Heisenberg group by the right action of the integral Heisenberg subgroup on itself, ie. we make the identification
\begin{align*}
    \begin{pmatrix}
    1 & a & c\\
    0 & 1 & b \\
    0 & 0 & 1
    \end{pmatrix}\sim \begin{pmatrix}
    1 & a+n & c+k+am\\
    0 & 1 & b+m \\
    0 & 0 & 1
    \end{pmatrix}\,,
\end{align*}
where $a\,,b\,,c\in\mathbb{R}$ and $n\,,m\,,k\in\mathbb{Z}\,.$ The projection onto $S^1$ is given by mapping to $b\,.$
\\\\The fiberwise product associated to the bundle $H_M\to S^1$ is given by
\begin{align*}
      \begin{pmatrix}
    1 & a & c\\
    0 & 1 & b \\
    0 & 0 & 1
    \end{pmatrix}\cdot  \begin{pmatrix}
    1 & a' & c'\\
    0 & 1 & b \\
    0 & 0 & 1
    \end{pmatrix}=  \begin{pmatrix}
    1 & a+a' & c+c'\\
    0 & 1 & b \\
    0 & 0 & 1
    \end{pmatrix}\,.
\end{align*}
See Example~\ref{module example} for more.
\end{exmp}
\begin{definition}
Let $M\,,N$ be $G$-modules. A morphism $f:M\to N$ is a morphism of the underlying groupoids such that if $s(g)=s(m)\,,$ then $f(g\cdot m)=g\cdot f(m)\,.$
$\blacksquare$\end{definition}
\begin{proposition}
Let $M\to X$ be a family of abelian groups. Then $H^1(X,\mathcal{O}(M))$ classifies principal $M$-bundles over $X$ for which $\rho=\pi\,.$
\end{proposition}
Before concluding this section we will make a remark on notation:
\theoremstyle{definition}\begin{rmk}Given a family of abelian groups $E\overset{\pi}{\to} Y\,,$ we can form its sheaf of sections, which as previously stated we denote by $\mathcal{O}(E)\,.$ In addition, given a map $f:X\to Y$ we get a family of abelian groups on $X\,,$ given by $f^*E=X{\mathop{\times}}_Y E\,.$
\end{rmk}
\subsection{Sheaves on Lie Groupoids and Stacks}
In this section we discuss the relationship between sheaves on $[G^0/G]\,,$ sheaves on $\mathbf{B}^\bullet G$ and $G$-modules ($\mathbf{B}^\bullet G$ is the nerve of $G\,,$ see appendix~\ref{appendix:Cohomology of Sheaves on Stacks} for more).
\subsubsection{Sheaves: Lie Groupoids to Stacks}\label{sheaves on stacks}
Here we discuss how to obtain a sheaf on the stack $[G^0/G]$ from a $G$-module.
\\\\Let $M$ be a $G$-module for $G\rightrightarrows G^0\,.$ We obtain a sheaf on $[G^0/G]$ as follows: consider the object of $[G^0/G]$ given by 
\[
\begin{tikzcd}
P \arrow{r}{\rho}\arrow{d}{\pi} & G^0 \\
X
\end{tikzcd}
\]
We can form the action groupoid $G\ltimes P$ and consider the $(G\ltimes P)$-module given by $\rho^*M\,.$ To $P$ we assign the abelian group $\Gamma_\text{inv}(\rho^*\mathcal{O}(M))$ (ie. the sections invariant under the $G\ltimes P$ action). To a morphism between objects of $[G^0/G]$ the functor just assigns the set-theoretic pullback. This defines a sheaf on $[G^0/G]\,,$ denoted $\mathcal{O}(M)_{[G^0/G]}\,.$

\subsubsection{Sheaves: Stacks to Lie Groupoids}\label{stacks to lie groupoids}
Here we discuss how to obtain a sheaf on $\mathbf{B}^\bullet G$ from a sheaf on $[G^0/G]\,,$ and we define the cohomology of a groupoid with coefficients taking values in a module.
\\\\Let $G\rightrightarrows G^0$ be a Lie groupoid and let $\mathcal{S}$ be a sheaf on $[G^0/G]\,.$ Consider the object of $[G^0/G]$ given by 
\[
\begin{tikzcd}
P \arrow{r}{\rho}\arrow{d}{\pi} & G^0 \\
X
\end{tikzcd}
\]
We can associate to each open set $U\subset X$ the object of $[G^0/G]$ given by
\[
\begin{tikzcd}
P\vert_U \arrow{r}{\rho}\arrow{d}{\pi} & G^0 \\
U
\end{tikzcd}
\]
We get a sheaf on $X$ by assigning to $U\subset X$ the abelian group $\mathcal{S}(P\vert_U)\,.$ 
\\\\Now for all $n\ge 0\,,$ the spaces $\mathbf{B}^{n} G$ are canonically identified with $G$-principal bundles, by identifying $\mathbf{B}^{n} G$ with the object of $[G^0/G]$ given by

\begin{equation}\label{canonical object}
\begin{tikzcd}[row sep=large, column sep=large]
 \mathbf{B}^{n+1} G \arrow{r}{d_{1,1}\circ p_{n+1}}\arrow{d}{d_{n+1,n+1}} & G^0 \\
\mathbf{B}^{n} G
\end{tikzcd}
\end{equation}
where $p_{n+1}$ is the projection onto the $(n+1)^{\text{th}}$ factor. Hence given a sheaf $\mathcal{S}$ on $[G^0/G]$ we obtain a sheaf on $\mathbf{B}^{n} G\,,$ for all $n\ge 0\,,$ denoted $\mathcal{S}(\mathbf{B}^{n} G)\,,$ and together these form a sheaf on $B^\bullet G\,.$ Furthermore, given a $G$-module $M$ we have that \begin{align*}
    \mathcal{O}(M)_{[G^0/G]}({G^0})\cong \mathcal{O}(M)\,.
\end{align*}
Moreover, we have the following lemma:
\begin{lemma}
Let $M$ be a $G$-module. Then the sheaf on $\mathbf{B}^\bullet G$ given by $\mathcal{O}(M)_{[G^0/G]}(\mathbf{B}^{\bullet}G)\,,$ is isomorphic to the sheaf of sections of the simplicial family of abelian groups given by
\begin{align*}
    \mathbf{B}^\bullet(G\ltimes M)\to \mathbf{B}^\bullet G\,.
\end{align*}
\end{lemma}
\theoremstyle{definition}\begin{definition}
Let $G\rightrightarrows G^0$ be a Lie groupoid and let $M$ be a $G$-module. We define 
\[
H^*(G,M):=H^*(\mathbf{B}^{\bullet}G,\mathcal{O}(M)_{[G^0/G]}(\mathbf{B}^{\bullet}G))\,.
\]
$\blacksquare$\end{definition}
\theoremstyle{definition}\begin{rmk}[See~\cite{Kai}]\label{morita inv}
Let $G\rightrightarrows G^0$ and $K\rightrightarrows K^0$ be Lie groupoids and let $\phi:G\to K$ be a Morita morphism. Then the pullback $\phi^*$ induces an equivalence of categories \[
\phi^*:[K^0/K]\to[G^0/G]\,.
\]
Furthermore, let $\mathcal{S}$ be a sheaf on $[G^0/G]\,.$ Then the pushforward sheaf $\phi_*\mathcal{S}:=\mathcal{S}\circ\phi^*$ is a sheaf on $[K^0/K]$ and we have a natural isomorphism 
\[
H^*(\mathbf{B}^{\bullet}G,\mathcal{S}(\mathbf{B}^{\bullet}G))\cong H^*(\mathbf{B}^{\bullet}K,\phi_*\mathcal{S}(\mathbf{B}^{\bullet}K))\,.
\]
\end{rmk}
\subsection{Godement Construction for Sheaves on Stacks}
Here we discuss a version of the Godement resolution for sheaves on stacks, and we show how it can be used to compute cohomology.
\begin{definition}
Let $G\rightrightarrows G^0$ be a Lie groupoid and let $\mathcal{S}$ be a sheaf on $[G^0/G]\,.$ We define the Godement resolution of $\mathcal{S}$ as follows:
Consider the object of $[G^0/G]$ given by
\[
\begin{tikzcd}
P \arrow{r}{\rho}\arrow{d}{\pi} & G^0 \\
X
\end{tikzcd}
\]
and consider the corresponding sheaf on $X$ (see~\Cref{stacks to lie groupoids}), denoted by $\mathcal{S}(X)\,.$
We can then consider, for each $n\ge 0\,,$ the $n^\text{th}$ sheaf in the Godement resolution of $\mathcal{S}(X)\,,$ denoted $\mathbb{G}^n(\mathcal{S}(X))\,,$ and to $P$ we assign the abelian group $\Gamma(\mathbb{G}^n(\mathcal{S}(X)))\,.$ These define sheaves on $[G^0/G]$ which we denote by $\mathbb{G}^n(\mathcal{S})\,.$ 
$\blacksquare$\end{definition}
For a sheaf $\mathcal{S}$ on $[G^0/G]$ we obtain a resolution by using $\mathbb{G}^\bullet(\mathcal{S})$ in the following way: 
\begin{align*}
    \mathcal{S}\xhookrightarrow{}\mathbb{G}^1(\mathcal{S})\to \mathbb{G}^2(\mathcal{S})\to\cdots\,.
    \end{align*}
The sheaves $ \mathbb{G}^n(\mathcal{S})$ are not in general acyclic on stacks, however the sheaves $\mathbb{G}^n(\mathcal{S})(\mathbf{B}^m G)$ are acyclic on $\mathbf{B}^m G$ and hence can be used to compute cohomology (see Theorem~\ref{stack groupoid cohomology} and Remark~\ref{double complex}).
\subsection{Examples}\label{important examples}
The constructions in the previous sections will be important in Section~\ref{van Est map} when defining the van Est map; it is crucial that modules define sheaves on stacks in order to use the Morita invariance of cohomology. Here we exhibit examples of the constructions from the previous sections which will be used in Section~\ref{van Est map}. 
\begin{proposition}\label{example surj sub}Let $f:Y\to X$ be a surjective submersion, and consider the submersion groupoid $Y{\mathop{\times}}_f Y\rightrightarrows Y\,.$ This groupoid is Morita equivalent to the trivial $X\rightrightarrows X$ groupoid, hence their associated stacks, $[Y/(Y{\mathop{\times}}_f Y)]$ and $[X/X]\,,$ are categorically equivalent. 
\end{proposition}
We now describe the functor $f^*:[X/X]\to[Y/(Y{\mathop{\times}}_f Y)]$ which gives this equivalence:
\\\\An $X\rightrightarrows X$ principal bundle is given by a manifold $N$ together with a map $\rho:N\to X$ (the $\pi$ map here is the identity map $N\to N)\,.$ To such an object, we let $f^*(N,\rho)=N\sideset{_\rho}{_{f}}{\mathop{\times}} Y\,.$ This is a $Y{\mathop{\times}}_f Y$ principal bundle in the following way:
\[\begin{tikzcd}
N\sideset{_\rho}{_{f}}{\mathop{\times}} Y \arrow{r}{\rho = p_2}\arrow{d}{\pi = p_1} & Y \\
N
\end{tikzcd}
\]
The functor $f^*$ is an equivalence of stacks. 
\\\\Now suppose we have a sheaf $\mathcal{S}$ on the stack $[Y/(Y{\mathop{\times}}_f Y)]\,,$ then we obtain a sheaf on $[X/X]$ by using the pushforward of $f\,,$ ie. to an object $(N,\rho)\in [X/X]$ we associate the abelian group $f_*\mathcal{S}(N,\rho):=\mathcal{S}(f^*(N,\rho))\,.$ We then obtain a sheaf on the simplicial space $\mathbf{B}^\bullet(X\rightrightarrows X)$ as follows: First note that $\mathbf{B}^n(X\rightrightarrows X)=X$ for all $n\ge 0\,,$ so the sheaves are the same on all levels. Now let $U\xhookrightarrow{\iota} X$ be open. Then $(U,\iota)\in [X/X]\,,$ so to this object we assign the abelian group $f_*\mathcal{S}(U,\iota)\,.$
\begin{proposition}\label{key example}Suppose $M$ is a $Y{\mathop{\times}}_f Y$-module and $f$ has a section $\sigma:X\to Y\,.$ We then obtain a sheaf (and its associated Godement sheaves) on $[X/X]\,,$ and in particular we obtain a sheaf (and its associated Godement sheaves) on $X\in [X/X]\,,$ which we describe as follows:
\end{proposition}
We use the notation in Proposition~\ref{example surj sub}. We have that
\[f^*(U,\iota)=\begin{tikzcd}
U_{\iota}{\mathop{\times}}_f Y \arrow{r}{\rho= p_2}\arrow{d}{\pi= p_1} & Y \\
U
\end{tikzcd}
=\begin{tikzcd}
Y\vert_U \arrow{r}{}\arrow{d}{f} & Y \\
U
\end{tikzcd}
\]
We then see that $\,\Gamma_\text{inv}(\rho^*\mathcal{O}(M))\cong \Gamma(\sigma\vert_U^*\mathcal{O}(M))\,,$ 
hence the sheaf we get on $X$ is simply $\sigma^*\mathcal{O}(M)\,.$ Furthermore, the sheaves we get on $X$ by applying the Godement construction to $\mathcal{O}(M)_{[Y/(Y{\mathop{\times}}_f Y)]}$ are simply $\mathbb{G}^\bullet(\sigma^*\mathcal{O}(M))\,.$
\begin{lemma}\label{acyclic edit}
Suppose we have a sheaf $\mathcal{S}$ on the stack $[Y/(Y{\mathop{\times}}_f Y)]\,,$ then the associated Godement sheaves $\mathbf{G}^\bullet(\mathcal{S})$ are acyclic.
\end{lemma}
\begin{proof}
This follows from the fact that $[Y/(Y{\mathop{\times}}_f Y)]\,,$ is Morita equivalent to $[X/X]\,,$ since cohomology is invariant under Morita equivalence of stacks, and the fact that the Godement sheaves on a manifold are acyclic. 
\end{proof}
\theoremstyle{definition}\begin{rmk}
Let $X$ be a manifold and let $X\rightrightarrows X$ be the trivial Lie groupoid. Let $\mathcal{S}$ be a sheaf on $[X/X]\,.$ Then we recover the usual cohomology: 
\[
H^*(\mathbf{B}^\bullet X,\mathcal{S}(\mathbf{B}^\bullet X))
=H^*(X,\mathcal{S}(X))\,.
\]
This will be important in computing the cohomology of submersion groupoids, since they are Morita equivalent to trivial groupoids.
\end{rmk}
\section{Chevalley-Eilenberg Complex for Modules}\label{Chevalley}
In this section we review the Chevalley-Eilenberg complex associated to a representation of a Lie algebroid. Then we generalize Lie
algebroid representations to Lie algebroid modules and define their Chevalley-Eilenberg complex. These will be used in Section~\ref{van Est map}.
\subsection{Lie Algebroid Representations}
\begin{definition}Let $\mathfrak{g}\overset{\pi}{\to} Y$ be a Lie algebroid, with anchor map $\alpha:\mathfrak{g}\to TY\,,$ and recall that $\mathcal{O}(\mathfrak{g})$ denotes the sheaf of sections of $\mathfrak{g}\overset{\pi}{\to} Y\,.$ A representation of $\mathfrak{g}$ is a vector bundle
 $E\to Y$ together with a map\begin{align*}
     \mathcal{O}(\mathfrak{g})\otimes\mathcal{O}(E)\to\mathcal{O}(E)\,,\,X\otimes s\mapsto L_X(s)
 \end{align*}
 such that for all open sets  $\,U\subset Y$ and for all $f\in \mathcal{O}_Y(U)\,,X\in \mathcal{O}(\mathfrak{g})(U)\,,
 s\in \mathcal{O}(E)(U)\,,$ we have that
 \begin{enumerate}
     \item $L_{fX}(s)=fL_X(s)\,,$
     \item $L_X(fs)=fL_X(s)+(\alpha(X)f)s\,,$
     \item $L_{[X,Y]}(s)=[L_X,L_Y](s)\,.$
 \end{enumerate}
 $\blacksquare$\end{definition}
\begin{definition}
Let $E$ be a representation of $\mathfrak{g}\,.$ Let $\mathcal{C}^n(\mathfrak{g},E)$ denote the sheaf of $E$-valued
 $n$-forms on 
 $\mathfrak{g}\,,$ ie. the sheaf of sections of $\Lambda^n \mathfrak{g}^*\otimes E\,.$ There is a canonical differential\footnote{Meaning in particular that $d_\text{CE}^2=0\,.$}
 \begin{align*}
     d_\text{CE}:\mathcal{C}^n(\mathfrak{g},E)\to \mathcal{C}^{n+1}(\mathfrak{g},E)\,,\,n\ge0
 \end{align*}
 defined as follows: let $\omega\in \mathcal{C}^n(\mathfrak{g},E)(V) $ for some open set $V\,.$ Then for $X_1\,,\ldots \,, X_{n+1}\in \pi^{-1}(m)\,,\,m\in V\,,$ choose local extensions $\mathbf{X}_1\,,\ldots\,,\mathbf{X}_{n+1}$ of these vectors, ie. choose
 \begin{align*}
     p\mapsto\mathbf{X}_1(p)\,,\ldots\,,p\mapsto\mathbf{X}_{n+1}(p)\in \mathcal{O}(\mathfrak{g})(U)\,,
     \end{align*}
     for some open set $U$ such that $m\in U\subset V\,,$ and such that $\mathbf{X}_i(m)=X_i$ for all $1\le i\le n+1)\,.$ Then let
 \begin{align*}
     d_{\text{CE}}\omega(X_1\,,\ldots\,,X_{n+1})&=\sum_{i<j}(-1)^{i+j-1}\omega([\mathbf{X}_i,\mathbf{X}_j],\mathbf{X}_1,\ldots ,\hat{\mathbf{X}}_i,\ldots
     , \hat{\mathbf{X}}_j,\ldots , \mathbf{X}_{n+1})
\vert_{p=m}     \\& +\sum_{i=1}^{n+1}(-1)^i L_{\mathbf{X}_i}(\omega(\mathbf{X}_i,\ldots , \hat{\mathbf{X}}_i ,\ldots , \mathbf{X}_{n+1}))\vert_{p=m}\,.
 \end{align*}
This is well-defined and independent of the chosen extensions.
 $\blacksquare$\end{definition}
\subsection{Lie Algebroid Modules}
We will now define Lie algebroid modules and define their Chevalley-Eilenberg complexes; these will look like the Chevalley-Eilenberg complexes associated to representations, except for possibly in degree zero (though representations will be seen to be special 
cases of Lie algebroid modules).
\begin{definition}
Let $\mathfrak{g}\to Y$ be a Lie algebroid, and let $M$ be a family of  abelian groups, with Lie algebroid $\mathfrak{m}$ and exponential map $\exp:\mathfrak{m}\to M\,.$\footnote{Note that $\mathfrak{m}$ is just a vector bundle and the exponential map is given by the fiberwise exponential map taking a Lie algebra to its corresponding Lie group.} Then a $\mathfrak{g}$-module structure on $M$ is given by the following: a $\mathfrak{g}$-representation structure on $\mathfrak{m}$ $($ie. a morphism  $\mathcal{O}(\mathfrak{g})\otimes\mathcal{O}(\mathfrak{m})\to\mathcal{O}(E)\,,\,X\otimes s\mapsto L_X(s))\,,$ together with a morphism of sheaves \begin{align*}
     \mathcal{O}(\mathfrak{g})\otimes_{\mathbb{Z}}\mathcal{O}(M)\to\mathcal{O}(\mathfrak{m})\,,\,X\otimes_{\mathbb{Z}} s\mapsto \tilde{L}_X(s)
 \end{align*}
  such that for all open sets  $\,U\subset Y$ and for all $f\in \mathcal{O}_Y(U)\,,X\in \mathcal{O}(\mathfrak{g})(U)\,,
 s\in \mathcal{O}(M)(U)\,,\sigma\in \mathcal{O}(\mathfrak{m})(U)\,, $ we have that
 \begin{enumerate}
     \item $\tilde{L}_{fX}(s)=f\tilde{L}_X(s)\,,$
     \item $\tilde{L}_{[X,Y]}(s)=(L_X\tilde{L}_Y-L_Y\tilde{L}_X)(s)\,,$
     \item $\tilde{L}_X(\exp{\sigma})=L_X(\sigma)\,.$
 \end{enumerate}
If $M$ is endowed with such a structure we call it a $\mathfrak{g}$-module.
$\blacksquare$\end{definition}
\begin{comment}Let $G\rightrightarrows X$ be a Lie groupoid and let $M$ be a $G$-module, and $\mathfrak{m}$ the associated representation, ie. a representation such that if $v\in \mathfrak{m}_x$ and $g\in G$ satisfies $s(g)=x\,,$ then 
\begin{align*}
    g\exp{v}=\exp{gv}\,.
\end{align*}
\end{comment}
\begin{definition}\label{forms}Let $\mathfrak{g}\to X$ be a Lie algebroid and let $M$ be a $\mathfrak{g}$-module. We then define sheaves on $X\,,$ called ``sheaves of $M$-valued forms", as follows: let
\begin{align*}
&\mathcal{C}^0(\mathfrak{g},M)=\mathcal{O}(M)\,,
\\&   \mathcal{C}^n(\mathfrak{g},M)=\mathcal{O}(\Lambda^n \mathfrak{g}^*\otimes \mathfrak{m})\,,\;n> 0\,.
\end{align*}
\begin{comment}Furthermore, let $d_{\text{CE}}$ denote the usual Chevalley-Eilenberg differentials for $\mathcal{C}^\bullet(\mathfrak{g},\mathfrak{m})\,.$ Define a map \begin{align*}
    d_{\text{CE}}\log:\mathcal{C}^0(\mathfrak{g},M)\to\mathcal{C}^1(\mathfrak{g},M)
    \end{align*}
as follows: let $s\in \mathcal{C}^0(\mathfrak{g},M)(U)\,,$ and for each $x\in U$ choose an open set $V$ such that $x\in V\subset U\,,$ and such that on $V$ there exists an
$\tilde{s}:V\to\mathfrak{m}$ for which $s\vert_V=\exp{\tilde{s}}\,.$ Define 
\begin{align*}
    d_{\text{CE}}\log\,s\vert_V=d_{\text{CE}}\,\tilde{s}\in \mathcal{C}^1(\mathfrak{g},M)(V)\,.
    \end{align*}
This is well-defined by~\eqref{condition} and determines a local section $d_{\text{CE}}\log s\in \mathcal{C}^1(\mathfrak{g},M)(U)\,.$
\end{comment}
Furthermore, for $s\in\mathcal{O}(M)(U)\,,$ we define $d_\text{CE}\log f$ by $d_\text{CE}\log f(X):=\tilde{L}_X(s)\,.$ We then have a cochain complex of sheaves given by
\begin{equation}\label{CE}
    \mathcal{C}^0(\mathfrak{g},M)\xrightarrow{d_{\text{CE}}\log}\mathcal{C}^1(\mathfrak{g},M)\xrightarrow{d_\text{CE}}\mathcal{C}^2(\mathfrak{g},M)\xrightarrow{d_\text{CE}}\cdots\,.
\end{equation}
 $\blacksquare$\end{definition}
\begin{definition}
The sheaf cohomology of the above complex of sheaves is denoted by $H^*(\mathfrak{g},M)\,.$
$\blacksquare$\end{definition}
\begin{definition}
Let $M\,,N$ be $\mathfrak{g}$-modules. A morphism $f:M\to N$ is a morphism of the underlying families of abelian groups such that the induced map $df:\mathfrak{m}\to\mathfrak{n}$ satisfies $\tilde{L}_X(f\circ s)=df\circ \tilde{L}_X(s)\,,$ for all local sections $X$ of $\mathfrak{g}$ and $s$ of $M\,.$
$\blacksquare$\end{definition}

 \theoremstyle{definition}\begin{exmp}
Here we will show that the notion of $\mathfrak{g}$-modules naturally extends the notion of $\mathfrak{g}$-representations. Let $E$ be a representation of $\mathfrak{g}\,.$ By thinking of the fibers of $E$ as abelian groups it defines a family of  abelian groups. The exponential map $E\overset{\exp}{\to} E$ is the identity, hence its kernel is the zero section and $E$ naturally defines a $\mathfrak{g}$-module where $d_\text{CE}\log =d_\text{CE}\,.$ So the definition of a $\mathfrak{g}$-module and its Chevalley-Eilenberg complex recovers the definition of a $\mathfrak{g}$-representation and its Chevalley-Eilenberg complex given by Crainic in~\cite{Crainic}.
\end{exmp}
 \theoremstyle{definition}\begin{exmp}\label{representation}
The group of isomorphism classes of $\mathfrak{g}$-representations on complex line bundles is isomorphic to $H^1(\mathfrak{g},\mathbb{C}^*_M)\,,$ where $\mathbb{C}^*_M$ is the $\mathfrak{g}$-module for which $\tilde{L}_X s=\mathrm{dlog}\,s(\alpha(X))\,,$ for a local section $s$ of $\mathbb{C}^*_M\,.$ The corresponding statement holds for real line bundles, with $\mathbb{C}^*_M$ replaced by $\mathbb{R}^*_M\,.$
\end{exmp}
\theoremstyle{definition}\begin{exmp}(Deligne Complex)
Let $X$ be a manifold and  $\mathfrak{g}=TX\,.$ Then letting $M=\mathbb{C}^*_X\,,$ we have that $\mathfrak{m}=\mathbb{C}_X$ naturally carries a representation of $TX\,,$ ie. where the differentials are the de Rham differentials. Letting $\exp:\mathfrak{m}\to M$ be the usual exponential map, it follows that $M$ is a $\mathfrak{g}$-module, and in fact the complex~\eqref{CE} in this case is known as the Deligne complex.
\end{exmp}
For a less familiar example we have the following:
\theoremstyle{definition}\begin{exmp}\label{module example}
Consider the space $S^1$ and the group $\mathbb{Z}/2\mathbb{Z}=\{-1,1\}\,.$ This group is contained in the automorphism  groups of $\mathbb{Z}\,,\mathbb{R}$ and $\mathbb{R}/\mathbb{Z}\,,$ hence we get nontrivial families of abelian groups over $S^1$  as follows (compare with Example~\ref{foag}): Let $A$ be any of the groups $\mathbb{Z}\,,\mathbb{R}\,,\mathbb{R}/\mathbb{Z}\,.$ Now cover $S^1$ in the standard way  using two open sets $U_0\,,U_1\,,$ and glue together the bundles $U_0{\mathop{\times}} A\,,U_1{\mathop{\times}} A$ with the transition functions
$-1\,,1$ on the two connected components of $U_0\cap U_1\,.$ Denote these families of abelian groups by $\tilde{\mathbb{Z}}\,,\tilde{\mathbb{R}}\,,\widetilde{\mathbb{R}/\mathbb{Z}}$ respectively. The space $\tilde{\mathbb{R}}$ is toplogically
the M\"{o}bius strip, and $\widetilde{\mathbb{R}/\mathbb{Z}}$ is topologically the Klein bottle.
\\\\Next, there is a canonical flat connection on these bundles of groups which is compatible with the fiberwise group structures, hence these families of abelian groups are modules for $\Pi_1(S^1)\,,$ the fundamental groupoid of $S^1\,.$ 
\\\\Furthermore, the $TS^1$-representation associated to the $TS^1$-module of $\tilde{\mathbb{Z}}$ is the rank $0$ vector bundle over $S^1\,,$ and the $TS^1$-representations associated to the $TS^1$-modules of $\tilde{\mathbb{R}}\,,\,\widetilde{\mathbb{R}/\mathbb{Z}}$ are isomorphic to the Mobius strip, ie. the line bundle obtained by gluing together $U_0{\mathop{\times}}\mathbb{R}\,,\,U_ 1{\mathop{\times}} \mathbb{R}$ using the same transition functions as discussed above. The Chevalley-Eilenberg differential, on each local trivialization $U_0{\mathop{\times}}\mathbb{R}\,,\,U_1{\mathop{\times}}\mathbb{R}\,,$ is just the de Rham differential.
\\\\The cohomology groups are $H^i(TS^1,\tilde{\mathbb{R}})=0$ in all degrees, and 
\[
H^i(TS^1,\widetilde{\mathbb{R}/\mathbb{Z}})=\begin{cases}
      \mathbb{Z}/2\mathbb{Z}, & \text{if}\ i= 0 \\
      0, & \text{if}\ i>0\,.
    \end{cases}
\]
\end{exmp}
\begin{theorem}\label{g-module}
Suppose $G\rightrightarrows G^0$ is a Lie groupoid. There is a natural functor
\begin{align*}
    F:G\text{-modules}\to \mathfrak{g}\text{-modules}\,.
\end{align*}
Furthermore, if $G$ is source simply connected then this functor restricts to an equivalence of categories on the subcategories of $G$-modules and $\mathfrak{g}$-modules for which $\exp:\mathfrak{m}\to M$ is a surjective submersion.
\end{theorem}
\begin{proof}$\mathbf{}$
\\\\$\mathbf{1.}$ For the first part, let $M$ be a $G$-module and for $x\in G^0\,$ let $\gamma:(-1,1)\to G(x,\cdot)$ be a curve in the source fiber such that $\gamma(0)=\text{Id}(x)\,.$ We define 
    \begin{align*}
    \tilde{L}_{\dot{\gamma}(0)}r:=\frac{d}{d\epsilon}\Big\vert_{\epsilon=0}r(x)^{-1}[\gamma(\epsilon)^{-1}\cdot r(t(\gamma(\epsilon)))]
    \end{align*}
for a local section $r$ of $\mathcal{O}(M)\,.$ One can check that this is well-defined and that property 1 is satisfied. Now note that the action of $G$ on $M$ induces a linear action of $G$ on $\mathfrak{m}\,,$ and we get a $\mathfrak{g}$-representation on $\mathfrak{m}$ by defining
\begin{align*}
    L_{\dot{\gamma}(0)}\sigma:=\frac{d}{d\epsilon}\Big\vert_{\epsilon=0}\sigma(x)^{-1}[\gamma(\epsilon)^{-1}\cdot \sigma(t(\gamma(\epsilon)))]
    \end{align*}
for a local section $\sigma$ of $\mathfrak{m}\,.$ With these definitions property 2 is satisfied. 
\\\\Now note that this action of $G$ on $\mathfrak{m}$ preserves the kernel of $\exp:\mathfrak{m}\to M\,.$ Let $\sigma$ be a local section of $\mathfrak{m}$
around $x$ such that
$\exp{\sigma}=e\,.$ Then 
\begin{align*}
    L_{\dot{\gamma}(0)}\sigma=\frac{d}{d\epsilon}\Big\vert_{\epsilon=0}\sigma(x)^{-1}[\gamma(\epsilon)^{-1}\cdot \sigma(t(\gamma(\epsilon)))]\,,
    \end{align*}
and since the $G$-action preserves the kernel of 
$\exp\,,$ which is discrete, we have that 
\begin{align*}
    \gamma(\epsilon)^{-1}\cdot \sigma(t(\gamma(\epsilon)))=\sigma(x)\,,
    \end{align*}
hence $L_{\dot{\gamma}(0)}(\sigma)=0\,,$ therefore $L(\sigma)=\tilde{L}(\exp{\sigma})=0\,,$ from which property 3 follows. Since it can be seen that morprhisms of $G$-modules induce morphisms of $\mathfrak{g}$-modules, this completes the proof.
\\\\$\mathbf{2.}$ For the second part, let $M$ be a $\mathfrak{g}$-module for which $\exp:\mathfrak{m}\to M$ is a surjective submersion, and suppose $G$ is source simply connected. Then in particular $\mathfrak{m}$ is a $\mathfrak{g}$-representation, and it is known that for source simply connected groupoids Rep$(G)\cong\text{Rep}(\mathfrak{g})\,,$ hence $\mathfrak{m}$ 
integrates to a $G$-representation. Property 3 implies that the $G$-action preserves the kernel of $\exp\,,$
hence the action of $G$ on $\mathfrak{m}$ descends to $M\,.$ More explicitly: let $g\in G(x,y)$ and let $m\in M_x\,,$ ie. the source fiber of $M$ over $x\,.$ Let 
$\tilde{m}\in\mathfrak{m}_x$ be such that $\exp{\tilde{m}}=m$ and define
\begin{align*}
    g\cdot m=\exp{(g\cdot{\tilde{m}})}\,.
\end{align*}
This is well-defined since the action of $G$ preserves the kernel of $\exp\,.$ Hence the functor is essentially surjective. Now again using the fact that for source simply connected groupoids
Rep$(G)\cong\text{Rep}(\mathfrak{g})\,,$ it follows that the functor is fully faithful, and since it is also essentially surjective, this completes the proof.
 \end{proof}
\begin{comment}$\mathbf{}$
\theoremstyle{definition}\begin{rmk}One can use the tools developped in this section to define connetions on principal groupoid bundles. See Section~\ref{abelian extensions} for more.
\end{rmk}
\end{comment}
\begin{comment}\\\\Let $G\rightrightarrows X$ be a Lie groupoid and let $M$ be a $G$-module. Let $\mathfrak{m}$ be the corresponding $G$-representation.
Define $d_\text{CE}\log:\Gamma(M)\to \mathcal{C}^1 (\mathfrak{g},M)$ by letting, for $\gamma'_0\in T_sG$ with representative $\gamma_\epsilon\in G_s\,,$ for $\epsilon\in (-1,1)\,,$
\begin{align*}
    d_\text{CE}\log f(\gamma'_0):=f(\gamma_0)^{-1}\left\frac{d}{d\epsilon}\right\vert_{\epsilon=0}[\gamma_\epsilon^{-1}f(t(\gamma_\epsilon))]\,.
\end{align*}
\end{comment}
\begin{comment}Alternatively, given a section $f\in\Gamma(M)\,,$ at each point $x\in X$ we can choose a logarithm of $f$ on a small neigbourhood $U\ni x$ and get a local section $\log{f}$ of $\mathfrak{m}\to X\,.$ We then define on $U$ 
\begin{align*}
     d_\text{CE}\log f(X):=d_\text{CE}(\log{f})(X)\,,
\end{align*}
where $X\in \mathfrak{g}$ and where on the right $d_\text{CE}$ is the differential coming from the corresponding representation of $\mathfrak{m}\,.$ The result is independent of the choice of logarithm and thus we get a well-defined global section of $\mathfrak{m}\to X\,.$
\end{comment}
\begin{comment}
\end{comment}Note that by composing with $\exp$ we obtain a natural map 
\begin{align*} 
V_{\mathfrak{g}}^*\overset{\exp}{\to} V_{G}^*\,.
\end{align*}
Claim: this map is an isomorphism. To see this, let $f\in V^*_{\mathfrak{g}}$ and suppose $\exp{f}=0\,.$ Let $v\in V$ and $t\in \mathbb{R}\,.$ Then $\exp{f(tv)}=0\implies \exp{tf(v)}=0.$ But taking $t\ne 0$ small enough such that $tf(v)\in U$ we get that $tf(v)=0\,,$ hence $f(v)=0\,.$
So $\exp$ is injective. 
\\\\Now let $f\in V^*_G\,.$ Let $v\in V\,.$ We must have $f(0)=0\,,$ so the map $[0,1]\to G$ defined by $t\mapsto f(tv)$ defines a path in $G$ starting at $0\,,$ therefore since $\mathfrak{g}$ is a covering space of $G\,,$ by the unique path lifting property there is a unique path $X_v:[0,1]\to\mathfrak{g}$ starting at $0$ which lifts $G\,.$ So define $X\in V^*_{\mathfrak{g}}$ by $X(v)=X_v(1)\,.$ Then $X(tv)=X_{tv}(1)$ and this defines a path lifting $f(tv)$ and starting at $0\,,$ hence $X(tv)=tX(v)\,.$ $X(a+b)=X_{a+b}(1)$

\\\\so by continuity for small enough $t\ne 0\in\mathbb{R}$ we have that $f(tv)\in \exp(U)\,.$ Hence there exists a unique $X_{tv}\in\mathfral{g}$ such that $\exp{X_{tv}}=f(tv)\,.$ Claim: $\exp{\frac{X_{tv}}{t}}=f(v).$
Now define $X\in V^*_{\mathfrak{g}}$ by $X(v)=\frac{X_{tv}}{t}\,.$

Let $\Lambda^0 V^*:=G\,.$ Let $\omega\in \Lambda^n V^*\,.$
\end{comment}

\section{van Est Map}\label{van Est map}
\subsection{Definition}
In this section we will discuss a generalization of the van Est map that appears in~\cite{Crainic}. It will be a map
$H^*(G,M)\to H^*(\mathfrak{g},M)\,,$ for a $G$-module $M\,,$ which will be an isomorphism up to a certain degree which depends on the connectivity of the source fibers of $G\,.$ Let us remark that one doesn't need to know the details of the map to understand the main theorem of this paper, \Cref{van Est image}, and if the reader wishes they may skip ahead to~\Cref{main theorem section}.
\\\\Given a groupoid $G\rightrightarrows G^0\,,$ $G$ naturally defines a principal $G$-bundle with the moment map given by $t\,,$ ie. the action is given by the left multiplication of $G$ on itself. Being consistent with the previous notation, we denote the resulting action groupoid by $G\ltimes G$ and note that it is isomorphic to $G_s{\mathop{\times}}_{s} G\rightrightarrows G\,,$ hence it is Morita equivalent to the trivial $G^0\rightrightarrows G^0$ groupoid. 
\begin{definition}
We let $\mathbf{E}^\bullet G:=\mathbf{B}^\bullet(G\ltimes G)\,.$
The simplicial map $\kappa:\mathbf{E}^\bullet G\to\mathbf{B}^\bullet G$ induced by the groupoid morphism $\pi_1:G\ltimes G\to G$ makes $\mathbf{E}^\bullet G$ into a simplicial principal $G$-bundle, and the fiber above $(g^1\,,\ldots\,,g^n)\in \mathbf{B}^n G$ is $t^{-1}(s(g_n))\,.$ 
$\blacksquare$\end{definition}
\theoremstyle{definition}\begin{rmk}Note that $G\ltimes G$ is a groupoid object in $[G^0/G]\,,$ and as a principal $G$-bundle it is the canonical object associated to $G$ via diagram~\ref{canonical object} 
\end{rmk}
\begin{comment} ie. it is  isomorphic to the object given by
\\\begin{center}
    \begin{tikzcd}
G^2\arrow{r}{}\arrow{d}{} & G \\
G^0
\end{tikzcd}
\end{center}
Now given a $G$-module $M\,,$ $\kappa^*M$ is a $(G\ltimes G)$-module, and its global sections are isomorphic to the abelian group associated to $G\ltimes G$ when considered as an object of $[G^0/G]\,.$ 
\end{comment}
\begin{definition}
Let $\Omega_{\kappa\;q}^p(\kappa^*M)$ denote the sheaf of sections of $\Lambda^p T^*_\kappa\, \mathbf{E}^qG\,(\kappa^*M)\,,$ the $\kappa$-foliated covectors taking values in $\kappa^*M\,.$ Succinctly, from $M$ we get a family of abelian groups on $\mathbf{B}^q G\,,$ given by \[
\mathbf{B}^q(G\ltimes M)\to \mathbf{B}^qG\,,
\]
which we denote by $M_{\mathbf{B}^qG}\,;$ we then have that $\kappa^*M_{\mathbf{B}^q G}$ is a module for the submersion groupoid 
\[\mathbf{E}^qG\times_{\mathbf{B}^qG} \mathbf{E}^qG\rightrightarrows \mathbf{E}^qG\,,
\]
and $\Omega_{\kappa\;q}^p(\kappa^*M)$ is the sheaf of $\kappa^*M_{\mathbf{B}^q G}$-valued p-forms associated to the corresponding Lie algebroid module (see~\Cref{forms}). Explicitly, $\Omega_{\kappa\;q}^0(\kappa^*M)$ is the sheaf of sections of~\footnote{Really, we should write $\kappa^*M_{\mathbf{B}^q G}\to\mathbf{E}^q G\,,$ but for notational simplicitly we suppress the subscript.}
\[
\kappa^*M\to \mathbf{E}^q G\,,
\]
and for $p\ge 1\,,$ $\Omega_{\kappa\;q}^p(\kappa^*M)$ is the sheaf of $\kappa$-foliated $p$-forms taking values in $\kappa^*\mathfrak{m}\,.$ There is a differential 
\begin{equation}
\Omega_{\kappa\;q}^0(\kappa^*M)\xrightarrow{\text{dlog}} \Omega_{\kappa\;q}^1(\kappa^*M)
\end{equation}
which is defined as follows: let $U$ be an open set in $\mathbf{E}^q G$ and let $X_g$ be a vector tangent to a $\kappa$-fiber at a point $g\in U\,.$ Let $f\in \Omega_{\kappa\;q}^0(\kappa^*M)(U)\,.$ Define $\text{dlog}\,f\in \Omega_{\kappa\;q}^1(\kappa^*M)(U)$ by 
\[
\text{dlog}\,f\,(X_g)=f(g)^{-1}f_*(X_g)\,,
\]
where in order to identify this with a point in $\kappa^*\mathfrak{m}_g$ we are implicity using the canonical identification of $\kappa^*M_g$ with $\kappa^*M_{g'}$ for any two points $g\,,g'$ in the same $\kappa$-fiber (here $\kappa^*M_g$ is the fiber of $\kappa^*M$ over $g)\,.$ We also use the canonical identification of $\kappa^*\mathfrak{m}_g$ with $\kappa^*\mathfrak{m}_{g'}$ for any two points $g\,,g'$ in the same $\kappa$-fiber to define the differentials for $p>0:$
\begin{equation}
\Omega_{\kappa\;q}^p(\kappa^*M)\overset{\text{d}}{\to}\Omega_{\kappa\;q}^{p+1}(\kappa^*M)\,.
\end{equation}
$\blacksquare$\end{definition}
$\mathbf{}$
\begin{theorem}
There is an isomorphism
\begin{equation*}
Q:H^*(\mathbf{E}^\bullet G\,,\kappa^{-1}\mathcal{O}(M))\to H^*(\mathfrak{g}\,,M)\,.
\end{equation*}
\end{theorem}
\begin{proof}Form the sheaf $\kappa^{-1}\mathcal{O}(M)$ on $\mathbf{E}^\bullet G\,.$ This sheaf is not in general a sheaf on the stack $[G/(G\ltimes G)]\,,$ but it is resolved by sheaves on stacks
in the following way:\footnote{These are sheaves on stacks because \begin{equation*}
\Lambda^n T^*_\kappa G(\kappa^*M)\cong\Lambda^n T^*_tG(t^*M)
\end{equation*}
(where $t$ is the target map),
and the latter are $(G\ltimes G)$-modules.}
\begin{comment}
and furthermore the sheaf of sections of $\Lambda^p T^*_\kappa\, \mathbf{E}^qG\,(\kappa^*M)$ are isomorphic to the sheaf on $\mathbf{E}^qG$ associated to the $(G\ltimes G)$-module $\Lambda^n T^*_\kappa G(\kappa^*M)\,.$
\end{comment}
\begin{equation}\label{deligne1}
    \kappa^{-1}\mathcal{O}(M)_\bullet\xhookrightarrow{} \mathcal{O}(\kappa^*M)_\bullet\overset{\text{dlog}}{\to}\Omega_{\kappa\,\bullet}^1(\kappa^*M)
    \overset{\text{d}}{\to}\Omega_{\kappa\,\bullet}^2(\kappa^*M)\to\cdots\,.
\end{equation}
We let, for all $q\ge 0\,,$
\begin{align}\label{C}C^\bullet_q:=\mathcal{O}(\kappa^*M)_q\to\Omega_{\kappa\,q}^1(\kappa^*M)
    \to\Omega_{\kappa\,q}^2(\kappa^*M)\to\cdots\,.
    \end{align}
We can then take the Godement resolution of $C^\bullet_q$ and get a double complex for each $q\ge 0:$ \begin{align*}
C^\bullet_q\xhookrightarrow{} \mathbb{G}^0(C^\bullet_q)\to \mathbb{G}^1(C^\bullet_q)\to\cdots\,.
\end{align*}
%We then can form the triple complex $\text{Tot}(\Gamma(\mathbb{G}^\bullet(C^\bullet_\bullet)))$ and use %this to compute the cohomology $H^*(\mathbf{E}^\bullet G, \kappa^{-1}\mathcal{O}(M))\,.$
\\All of the sheaves $\mathbb{G}^p(C^r_\bullet)$ are sheaves on stacks, and it follows that these sheaves are acyclic (as sheaves on stacks) since $G\ltimes G\rightrightarrows G$ is Morita equivalent to a submersion groupoid, and \Cref{acyclic edit}. Hence $\mathbb{G}^p(C^r_\bullet)$ can be used to compute cohomology (see Remark~\ref{acyclic resolution}) and we have that
\begin{align*}
    H^*(\mathbf{E}^\bullet G, \kappa^{-1}\mathcal{O}(M))\cong  H^*(\text{Tot}(\Gamma_\text{inv}(\mathbb{G}^\bullet (C^\bullet_0))))\,.
\end{align*}
Now we have that 
\begin{align*}
    \Gamma_\text{inv}(\mathbb{G}^p(C^q_0)))=\Gamma(\mathbb{G}^p(i^*C^q_0))\,, 
\end{align*}
where $i:G^0\to G$ is the identity bisection. Since all of the differentials in~\eqref{C} preserve invariant sections, they descend to differentials on 
$\Gamma(\mathbb{G}^\bullet(i^*C^\bullet_0))\,,$ hence 
\begin{align*}
    H^*(\text{Tot}(\Gamma_\text{inv}(\mathbb{G}^\bullet (C^\bullet_0))))\cong H^*(\mathfrak{g},M)\,.
\end{align*}
\end{proof}
\begin{comment}
where 
\begin{align*}
    H^*(\mathfrak{g}\,,M)=H^*(\,\mathcal{C}^0(\mathfrak{g}\,,M)\overset{d_{\text{CE}}\text{log}}{\to}\mathcal{C}^1(\mathfrak{g}\,,M)\overset{d_{\text{CE}}}{\to}\mathcal{C}^2(\mathfrak{g}\,,M)\overset{d_{\text{CE}}}{\to}\cdots)\,,
\end{align*}
where $\mathcal{C}^0(\mathfrak{g},M)=\mathcal{O}(M)$ and $\mathcal{C}^n(\mathfak{g},M)$ is the sheaf of sections of Lie algebroid $n$-forms taking values in $\mathfrak{m}\,,$ for $n>0\,,$ and where $d_{\text{CE}}$ are the usual Chevalley-Eilenberg differentials.
\end{comment}
\theoremstyle{definition}\begin{definition}\label{van Est}
Let $M$ be a $G$-module. We define a map 
\begin{equation*}
H^*(G\,,M)\to H^*(\mathfrak{g}\,,M)
\end{equation*}
given by the composition
\begin{equation*}
H^*(G\,,M)\overset{\kappa^{-1}}{\to} H^*(\mathbf{E}^\bullet G\,,\kappa^{-1}\mathcal{O}(M)_\bullet)\overset{Q}{\to}H^*(\mathfrak{g}\,,M)\,.
\end{equation*}
This is the van Est map; we denote it by $VE\,.$
\end{definition}
\theoremstyle{definition}\begin{rmk}
Taking $M=\mathbb{C}_{G^0}$ as a smooth abelian groups with the trivial $G$-action, the sheaves in the resolution of $\kappa^{-1}\mathcal{O}(M)$ in~\eqref{deligne1} are already acyclic (as sheaves on stacks). Hence our map coincides with the map \begin{align*}
    H^*(G\,,M)&\to H^*(\mathbf{E}^\bullet G\,,\kappa^{-1}\mathcal{O}(M)_\bullet)
    \\&=H^*(\Gamma_\text{inv}(\mathcal{O}(\kappa^*M)_0)\to\Gamma_\text{inv}(\Omega_{\kappa\,0}^1(M))
  \to\cdots)
  \to H^*(\mathfrak{g}\,,M)\,,
\end{align*}
which is the van Est map as described in \cite{Meinrenken}.
\end{rmk} 
\subsection{van Est for Truncated Cohomology}
In order to emphasize geometry on the space of morphisms rather than on the space of objects we perform a truncation. That is, we truncate the contribution of $G^0$ to $H^*(G,M)$ by considering instead the cohomology 
\[
H^*(\mathbf{B}^\bullet G,\mathcal{O}(M)^0_\bullet)\,,
\]
where $\mathcal{O}(M)^{0}_n=\mathcal{O}(M)_n$ for all $n\ge 1\,,$ and where $\mathcal{O}(M)^{0}_0$ is the trivial sheaf on $G^0\,,$ ie. the sheaf that assigns to every open set the group containing only the identity.
\\\\We define 
\[
H^*_0(G,M):=H^{*+1}(\mathbf{B}^\bullet G,\mathcal{O}(M)^{0}_\bullet)\,.
\]
There is a canonical map 
\[
H^*_0(G,M)\to H^{*+1}(G,M)\
\]
induced by the morphism of sheaves on $\mathbf{B}^\bullet G$ given by $\mathcal{O}(M)^{0}_\bullet\xhookrightarrow{}\mathcal{O}(M)_\bullet\,.$
Similarly, we can truncate $M$ from $H^*(\mathfrak{g},M)$ by considering instead 
\[
H^*_0(\mathfrak{g},M):=H^{*+1}(0\to\mathcal{C}^1(\mathfrak{g},M)\to\mathcal{C}^2(\mathfrak{g},M)\to\cdots)\,.
\]
Then in like manner there is a canonical map 
\[
H^*_0(\mathfrak{g},M)\mapsto H^{*+1}(\mathfrak{g},M)
\]
induced by the inclusion of the truncated complex into the full one.
\begin{comment}\\\\Now let $\mathcal{O}(M)^{0}_\bullet$ be the sheaf on $\mathbf{B}^\bullet(G^0\rightrightarrows G^0)$ such that $\widetilde{\mathcal{O}(M)}_0=\mathcal{O}(M)$ and $\widetilde{\mathcal{O}(M)}_n$ is the trivial sheaf for $n\ge 1\,.$ Let $i:\mathbf{B}^\bullet(G^0\rightrightarrows G^0)\to \mathbf{E}^\bullet G$ be the canonical embedding, and let $\tilde{\mathcal{O}(M)}_\bullet$ be the sheaf on $\mathbf{E}^\bullet G$ given by $\tilde{\mathcal{O}(M)}_n=\mathcal{O}(\kappa^*M)_n$ for $n\ge 1\,,$ and such that $\tilde{\mathcal{O}(M)}_0$ is the subsheaf of $\mathcal{O}(\kappa^*M)_0$ consisting of local sections which are the identity on $G^0\,.$ Then we have the following short exact sequence of sheaves:
\begin{align}\label{SES}
    0\to \tilde{\mathcal{O}(M)}_\bullet\to \mathcal{O}(\kappa^*M)_\bullet\to i_*\mathcal{O}(M)^0_\bullet\to 0\,.
\end{align}
Since $H^*(\mathbf{E}^\bullet G,\mathcal{O}(\kappa^*M)_\bullet)\cong H^*(\mathbf{E}^\bullet G,i_*\widetilde{\mathcal{O}(M)}_\bullet)\,,$
it follows from the long exact sequence associated to \eqref{SES} that $H^*(\mathbf{E}^\bullet G,\mathcal{O}(M)^+_\bullet)=0\,.$ The same is true for the Godement sheaves, that is for each $n\ge 0$ the following is a short exact sequence of sheaves:
\begin{align*}
    0\to \mathbb{G}^n(\mathcal{O}(M)^+_\bullet)\to \mathbb{G}^n(\mathcal{O}(\kappa^*M)_\bullet)\to \iota_*\mathbb{G}^n(\widetilde{\mathcal{O}(M)}_\bullet)\to 0\,,
\end{align*}
which in like manner implies that 
\end{comment}
\begin{theorem}\label{canonical edit}
There is a canonical map $VE_0$ lifting $VE\,,$ ie. such that the following diagram commutes:
\begin{equation}\label{diagram}
\begin{tikzcd}
H^*_0(G,M) \arrow{d}{}\arrow{r}{VE_0} & H^*_0(\mathfrak{g},M)\arrow{d} \\
H^{*+1}(G,M)\arrow{r}{VE} & H^{*+1}(\mathfrak{g},M) 
\end{tikzcd}
\end{equation}
where $H_0^*$ denotes truncated cohomology, as define above.
\end{theorem}
\begin{proof}
Consider the ``normalized" sheaf on $\mathbf{E}^\bullet G$ given by $\widehat{\mathcal{O}(\kappa^*M)}_\bullet\,,$ where $\widehat{\mathcal{O}(\kappa^*M)}_n=\mathcal{O}(\kappa^*M)_n$ for $n\ge 1\,,$ and where $\widehat{\mathcal{O}(\kappa^*M)}_0$ is the subsheaf of $\mathcal{O}(\kappa^*M)_0$ consisting of local sections which are the identity on $G^0\,.$ Then \[
H^*(\mathbf{E}^\bullet G,\mathbb{G}^n(\widehat{\mathcal{O}(\kappa^*M)}_\bullet))=0\,,
\] and in particular, $\mathbb{G}^n(\widehat{\mathcal{O}(\kappa^*M)}_\bullet)$ is acyclic.
$\mathbf{}$
\\\\Now consider the sheaf $\widehat{\kappa^{-1}\mathcal{O}(M)}_\bullet$ on $\mathbf{E}^\bullet G$ given by $\widehat{\kappa^{-1}\mathcal{O}(M)}_n=\kappa^{-1}\mathcal{O}(M)_n$ for $n\ge 1\,,$ and such that $\widehat{\kappa^{-1}\mathcal{O}(M)}_0$ is the subsheaf of $\kappa^{-1}\mathcal{O}(M)_0$ consisting of local sections which are the identity on $G^0\,.$
Then there is a canonical embedding $\kappa^{-1}\mathcal{O}(M)^0_\bullet\xhookrightarrow{}\widehat{\kappa^{-1}\mathcal{O}(M)}_\bullet\,,$ hence we get a map
\begin{align*}
    H^*(\mathbf{B}^\bullet G,\mathcal{O}(M)^0_\bullet)\to H^*(\mathbf{E}^\bullet G,\widehat{\kappa^{-1}\mathcal{O}(M)}_\bullet)\,.
\end{align*}
Now we have that the following inclusion is a resolution:
\begin{align*}
    \widehat{\kappa^{-1}\mathcal{O}(M)}_\bullet\xhookrightarrow{} \widehat{\mathcal{O}(\kappa^*M)}_\bullet\to\Omega_{\kappa\,\bullet}^1(M)
    \to\Omega_{\kappa\,\bullet}^2(M)\to\cdots\,.
\end{align*}
Then one can show that 
\begin{align*}
   & H^*(\mathbf{E}^\bullet G,\widehat{\kappa^{-1}\mathcal{O}(M)}_\bullet\to\Omega_{\kappa\,\bullet}^1(M)
    \to\Omega_{\kappa\,\bullet}^2(M)\to\cdots)
    \\&\cong H^{*}(\mathbf{E}^\bullet G,0\to\Omega_{\kappa\,\bullet}^1(M)
    \to\Omega_{\kappa\,\bullet}^2(M)\to\cdots)\,,
\end{align*}
and since $\Omega_{\kappa\,\bullet}^1(M)\to\Omega_{\kappa\,\bullet}^2(M)\to\cdots$ is a complex of sheaves on stacks, by a similar argument made when defining the van Est map in the previous section we get that 
\begin{align*}
    H^{*+1}(\mathbf{E}^\bullet G,0\to\Omega_{\kappa\,\bullet}^1(M)\to\Omega_{\kappa\,\bullet}^2(M)\to\cdots)\cong H^*_0(\mathfrak{g},M)\,.
\end{align*}
Then $VE_0$ is the map 
\begin{align*}
   & H^*_0(G,M)=H^{*+1}(\mathbf{B}^\bullet G,\mathcal{O}(M)^0_\bullet)\xrightarrow{\kappa^{-1}} H^{*+1}(\mathbf{E}^\bullet G,\kappa^{-1}\mathcal{O}(M)^0_\bullet)
   \\&\to H^{*+1}(\mathbf{E}^\bullet G,\widehat{\kappa^{-1}\mathcal{O}(M)}_\bullet)
   \,{\cong}\, H^{*+1}(\mathbf{E}^\bullet G,\widehat{\mathcal{O}(\kappa^*M)}_\bullet\to\Omega_{\kappa\,\bullet}^1(M)
    \to\cdots)
    \\&{\cong}\,H^{*+1}(\mathbf{E}^\bullet G,0\to\Omega_{\kappa\,\bullet}^1(M)
    \to\cdots)
    \xrightarrow{\cong} H^{*}_0(\mathfrak{g},M)\,.
\end{align*}
\end{proof}
\theoremstyle{definition}\begin{rmk}
The van Est map (including the truncated version) factors through a local van Est map defined on the cohomology of the local groupoid [see~\cite{Meinrenken}], ie. to compute the van Est map one can first localize the cohomology classes to a neighborhood of the identity bisection.
\end{rmk}
\subsection{Properties of the van Est Map}
In this section we discuss some properties of the van Est map; the main results pertain to its kernel and image.
\\\\Recall that given a sheaf $\mathcal{S}_\bullet$ on a (semi) simplicial space $X^\bullet\,,$ we calculate its cohomology by taking an injective resolution $0\to\mathcal{S}_\bullet\to \mathcal{I}^0_\bullet\to\mathcal{I}^1_\bullet\to\cdots$
and computing 
\begin{align*}
    H^*(\,\Gamma_\text{inv}(\mathcal{I}^0_0)\to\Gamma_\text{inv}(\mathcal{I}^1_0)\to\cdots)\,.
\end{align*}
By considering the natural injection $\Gamma_\text{inv}(\mathcal{I}_0^n)\hookrightarrow\Gamma(\mathcal{I}^n_0)$ we get a map
\begin{equation}\label{restriction}
    r:H^*(X^\bullet,\mathcal{S}_\bullet)\to H^*(X^0,\mathcal{S}_0)\,.
    \end{equation}
Similarly, for a cochain complex of abelian groups $\mathcal{A}^0\to \mathcal{A}^1\to\cdots$ there is a map 
\begin{align*}
H^*(\mathcal{A}^0\to \mathcal{A}^1\to\cdots)\overset{r}{\to} H^*(\mathcal{A}^0)\,.
\end{align*}
Using this, we have the following result, which gives an enlargement of diagram \eqref{diagram}:
\begin{lemma}\label{commute}
The following diagram is commutative:
\[
\begin{tikzcd}
H^{*}(G^0,\mathcal{O}(M))\arrow{r}{\delta^*}\arrow{d}{\parallel} &H^*_0(G,M)\arrow{r}{}\arrow{d}{VE_0} &H^{*+1}(G,M) \arrow{r}{r} \arrow{d}{VE} & H^{*+1}(G^0,\mathcal{O}(M))\arrow{d}{\parallel} \\
H^{*}(G^0,\mathcal{O}(M))\arrow{r}{d_{CE}\log}& H^*_0(\mathfrak{g},M)\arrow{r}{}& H^{*+1}(\mathfrak{g},M) \arrow{r}{r}& H^{*+1}(G^0,\mathcal{O}(M))
\end{tikzcd}
\]
\end{lemma}
\begin{lemma}\label{kernel}
Suppose that $X\overset{\pi}{\to} Y$ is a surjective submersion with $(n-1)$-connected fibers, for some $n>0\,,$ and with a section $\sigma\,.$  Consider an exact sequence of of families of abelian groups on $Y$ given by
\begin{align*}
    0\to Z\xrightarrow{}\mathfrak{m}\xrightarrow{\exp{}}M\,.
\end{align*} Let $\omega\in H^0(X,\Omega^n_{\pi}(\pi^*\mathfrak{m}))$ be closed (ie. $\omega$ is a closed, foliated $n$-form on $X$) and suppose that
\begin{equation}\label{in kernel}
\int_{S^n(\pi^{-1}(y))}\omega \in Z\;\;\text{     for all } y \in Y \text{ and all } S^n(\pi^{-1}(y))\,,
\end{equation}
where $S^n(\pi^{-1}(y))$ is an $n$-sphere contained in the source fiber over $y\,.$ 
 %Then letting $\alpha\in H^n(\mathbf{E}^\bullet G\,,
 %\kappa^{-1}\mathfrak{m}_\bullet)$ be such that such that $\alpha\mapsto\omega$ under the map 
 %\begin{align*}
%H^n(\mathbf{E}^\bullet G\,,
 %\kappa^{-1}\mathfrak{m}_\bullet)\to H^n(\mathfrak{g},\mathfrak{m})\,,
% \end{align*}
Let $[\omega]$ denote the class $\omega$ defines in $H^n(X,\pi^{-1}\mathcal{O}(\mathfrak{m}))\,.$ Then $\exp{[\omega]}=0\,.$ 
%under the map
%begin{align*}
  %  r:H^*(\mathbf{E}^\bullet G,\kappa^{-1}\mathcal{O}(M)_\bullet)\to H^*(\mathbf{E}^0G,\kappa^{-1}\mathcal{O}(M)_0)\,.
%\end{align*}
\begin{proof}
%Consider $r(\alpha)\,;$ we want to show that $\exp{r(\alpha)}=0\,.$ 
From Equation~\ref{in kernel} we know that $\exp{\omega}\vert_{\pi^{-1}(y)}=0$ for each $y\in Y\,,$ therefore since the source fibers of $X$ are $(n-1)$-connected, by Theorem~\ref{spectral theorem} we have that 
\begin{align*}
    \exp{[\omega]}=\pi^{-1}\beta
    \end{align*}
for some $\beta\in H^n(Y,\mathcal{O}(M))\,.$ Since $\pi\circ \sigma:Y\to Y$ is the identity this implies that
\begin{align*}
    \exp{\sigma^{-1}[\omega]}=\beta\,,
    \end{align*}
    but $\sigma^{-1}[\omega]=0$ since $\omega$ is a global foliated form. Hence $\beta=0\,,$ hence $\exp{[\omega]}=0\,.$
\end{proof}
\end{lemma}
\begin{corollary}\label{groupoid kernel}
Suppose that $G\rightrightarrows X$ is source $(n-1)$-connected for some $n>0\,.$ Consider an exact sequence of families of abelian groups  on $X$ given by
\begin{align*}
    0\xrightarrow{} Z\xrightarrow{}\mathfrak{m}\xrightarrow{\exp{}} M\,.
\end{align*} Let $\omega\in H^0(\mathcal{C}^0(\mathfrak{g},\mathfrak{m}))$ be closed (ie. it is a closed $n$-form in the Chevalley-Eilenberg complex) and suppose that
\begin{equation}\label{in kernel 2}
\int_{S^n(s^{-1}(x))}\omega \in Z\;\;\text{     for all } x \in X \text{ and all } S^n(s^{-1}(x))\,,
\end{equation}
where in the above we have left translated $\omega$ to a source-foliated $n$-form, and where $S^n(s^{-1}(x))$ is an $n$-sphere contained in the source fiber over $x\,.$
 %Then letting $\alpha\in H^n(\mathbf{E}^\bullet G\,,
 %\kappa^{-1}\mathfrak{m}_\bullet)$ be such that such that $\alpha\mapsto\omega$ under the map 
 %\begin{align*}
%H^n(\mathbf{E}^\bullet G\,,
 %\kappa^{-1}\mathfrak{m}_\bullet)\to H^n(\mathfrak{g},\mathfrak{m})\,,
% \end{align*}
Let $[\omega]$ denote the class $\omega$ defines in $H^n(\mathbf{E}^\bullet G,\kappa^{-1}\mathfrak{m})\,.$ Then $r(\exp{([\omega])})=0\,,$ where $r$ is as in Equation~\ref{restriction}. 
%under the map
%begin{align*}
  %  r:H^*(\mathbf{E}^\bullet G,\kappa^{-1}\mathcal{O}(M)_\bullet)\to H^*(\mathbf{E}^0G,\kappa^{-1}\mathcal{O}(M)_0)\,.
%\end{align*}
\end{corollary}
\begin{proof}
This follows directly from Lemma~\ref{kernel}.
\end{proof}
\subsection{Main Theorem}\label{main theorem section}
Before proving the main theoem of the paper, we will discuss translation of Lie algebroid objects: similarly to how one can translate Lie algebroid forms to differential forms along the source fibers, one can translate all Lie algebroid cohomology classes (eg. 1-dimensional Lie algebroid representations) to cohomology classes along the source fibers (in the case of a Lie algebroid representation, translation will result in a principal bundle with flat connection along the source fibers). 
We will describe it in degree 1, the other cases are similar. 
\theoremstyle{definition}\begin{definition}\label{translation}

Let $G\rightrightarrows G^0$ be a Lie groupoid and let $M$ be a $G$-module. Let $\{U_i\}_i$ be an open cover of $G^0$ and let $\{(h_{ij},\alpha_i\}_{ij}$ represent a class in $H^1(\mathfrak{g},M)$ (here the $h_{ij}$  are sections of $M$ over $U_i\cap U_j\,,$ and the $\alpha_i$ are Lie algebroid 1-forms taking values in $\mathfrak{m}$). Then on $\mathbf{B}^1G$ we get a class in foliated cohomology (foliated with respect to the source map), ie. a class in
\begin{align}
    H^1(\mathcal{O}(s^*M)\overset{\text{dlog}}{\to}\Omega_{s}^1(s^*M)
    \overset{\text{d}}{\to}\Omega_{s}^2(s^*M)\to\cdots)\,,
\end{align}
\begin{comment}Let $x\in G^0$ and let $M_x\,,\mathfrak{m}_x$ be the Lie group and Lie algebra given by the fibers over $x$ in $M$ and $\mathfrak{m}\,,$ respectively. Then on $s^{-1}(x)$ we get a class in \begin{align*}
H^1(s^{-1}(x),\mathcal{O}(M_x)\xrightarrow{\text{dlog}}\Omega^1\otimes\mathfrak{m}_x\to\Omega^2\otimes\mathfrak{m}_x\to\cdots)
\end{align*}
\end{comment}
defined as follows:
\begin{comment}we have a map $t:s^{-1}(x)\to X$ given by the target map, and we get an open cover of $s^{-1}(x)$
\end{comment}
we have an open cover of $\mathbf{B}^1G$ given by $\{t^{-1}(U_i)\}_i\,.$ We then get a principal $s^*M$-bundle over $\mathbf{B}^1G$ given by transition functions $t^*h_{ij}$ defined as follows: for $g\in t^{-1}(U_{ij})$ let \begin{align}
    t^*h_{ij}(g):=g^{-1}\cdot h_{ij}(t(g))\,.
\end{align}Similarly, we define foliated 1-forms $t^*\alpha_i$ on $t^{-1}(U_i)$ as follows: for $g\in t^{-1}(U_i)$ with $s(g)=x\,,$ and for $V_g\in T_g(s^{-1}(x))\,,$ let
\begin{align*}
t^*\alpha_i(V_g):=g^{-1}\cdot\alpha_{i}(R_{g^{-1}}V_g)\,,
\end{align*}
where $R_{g^{-1}}$ denotes translation by $g^{-1}\,.$ Then the desired class is given by the cocycle $\{(t^*h_{ij},t^*{\alpha_i})\}_i$ on the open cover $\{t^{-1}(U_i)\}\,.$ For each $x\in G^0\,,$ by restricting the cocycle to $s^{-1}(x)$ we also get a class in \begin{align*}
H^1(s^{-1}(x),\mathcal{O}(M_x)\xrightarrow{\text{dlog}}\Omega^1\otimes\mathfrak{m}_x\to\Omega^2\otimes\mathfrak{m}_x\to\cdots)\,.
\end{align*}
\begin{comment}In addition, performing the above construction for all $x\in X$ defines a class in the foliated cohomology of $G$ (foliated with respect to the source map, see~\Cref{van Est map}).
\end{comment}
\\Similarly, we can translate any class $\alpha\in H^{\bullet}(\mathfrak{g},M)$ to a class in
\begin{align}\label{class}
    H^{\bullet}(\mathcal{O}(s^*M)\overset{\text{dlog}}{\to}\Omega_{s}^1(s^*M)
    \overset{\text{d}}{\to}\Omega_{s}^2(s^*M)\to\cdots)\,,
\end{align}
and we denote 
this class by $t^*\alpha\,.$ Furthermore, for each $x\in G^0$ we obtain as class in
\begin{align*}
H^\bullet(s^{-1}(x),\mathcal{O}(M_x)\xrightarrow{\text{dlog}}\Omega^1\otimes\mathfrak{m}_x\to\Omega^2\otimes\mathfrak{m}_x\to\cdots)\,,
\end{align*} and we denote this class by $t_x^*\alpha\,.$ 
\\\\Alternatively, given a class $\alpha\in H_0^{\bullet}(\mathfrak{g},M)\,,$ we can translate this to a class in 
\begin{align}
    H^\bullet(0\to\Omega_{s}^1(s^*M)
    \overset{\text{d}}{\to}\Omega_{s}^2(s^*M)\to\cdots)\,,
\end{align}
and we denote this class by $t_0^*\alpha\,.$ In this case the notation $t^*\alpha$ will be used to mean the class obtained in in~\ref{class} by first viewing $\alpha$ as a class in $H^\bullet(\mathfrak{g},M)\,.$
$\blacksquare$
\end{definition}
\begin{proposition}
With the previous definition, we have the following commutative diagram:
\[
\begin{tikzcd}
H^\bullet_0(\mathfrak{g},M) \arrow{d}{}\arrow{r}{t_0^*} & H^{\bullet+1}(0\to\Omega_{s}^1(s^*M)
    \overset{\text{d}}{\to}\cdots)
\arrow{d} \\
H^{\bullet+1}(\mathfrak{g},M)\arrow{r}{t^*} 
&   H^{\bullet+1}(\mathcal{O}(s^*M)\to\Omega_{s}^1(s^*M)
    \overset{\text{d}}{\to}\cdots)
\end{tikzcd}
\]
$\blacksquare$
\end{proposition} 
The importance of the previous definition is due to the fact that given a class in $\alpha\in H^{\bullet}(\mathfrak{g},M)\,,$ the class $t^*\alpha$ defines a class in $H^{\bullet}(E^\bullet G\,,\kappa^{-1}\mathcal{O}(M))$ (or if $\alpha\in H^\bullet_0(\mathfrak{g},M)\,,$ then $t_0^*\alpha$ defines a class in $H^{\bullet}(E^\bullet G\,,\widehat{\kappa^{-1}\mathcal{O}(M)})\,,$ see~\Cref{van Est map}). 
We are now ready to state and prove the main theorem of the paper:
\begin{theorem}[Main Theorem]\label{van Est image}
Suppose $G\rightrightarrows G^0$ is source $n$-connected and that $M$ is a $G$-module fitting into the exact sequence
\[
0\to Z\to \mathfrak{m}\overset{exp}{\to} M\,,
\] where $\mathfrak{m}$ is the Lie algebroid of $M\,.$ Then the van Est map  $\,VE:H^*(G,M)\to H^*(\mathfrak{g},M)$ is an isomorphism in degrees $\le n$ and injective in degree $(n+1)\,.$ The image of $VE$ in degree $(n+1)$ are the classes $\alpha\in H^{n+1}(\mathfrak{g},M)$ such that for all $x\in G^0\,,$ the translated class $t_x^*\alpha$ (see Definition~\ref{translation}) is trivial in \begin{align*}
H^{n+1}(s^{-1}(x),\mathcal{O}(M_x)\xrightarrow{\text{dlog}}\Omega^1\otimes\mathfrak{m}_x\to\Omega^2\otimes\mathfrak{m}_x\to\cdots)\,.
\end{align*} The same statement holds for $VE_0:H_0^*(G,M)\to H_0^*(\mathfrak{g},M)$ with a degree shift, that is: the truncated van Est map  $VE_0:H_0^*(G,M)\to H_0^*(\mathfrak{g},M)$ is an isomorphism in degrees $\le n-1$ and injective in degree $n\,.$ The image of $VE_0$ in degree $n$ are the classes $\alpha\in H_0^{n}(\mathfrak{g},M)$ such that for all $x\in G^0\,,$ the translated class $t_x^*\alpha$ is trivial in \begin{align*}
H^{n+1}(s^{-1}(x),\mathcal{O}(M_x)\xrightarrow{\text{dlog}}\Omega^1\otimes\mathfrak{m}_x\to\Omega^2\otimes\mathfrak{m}_x\to\cdots)\,.
\end{align*}
In particular, let $\omega$ be a closed Lie algebroid $(n+1)$-form, ie.
\[
\omega \in \ker\big[\Gamma(\mathcal{C}^{n+1}(\mathfrak{g},M))\xrightarrow{d_\text{CE}}\Gamma(\mathcal{C}^{n+2}(\mathfrak{g},M))\big]\,.
\]
Then $[\omega]\in H^{n}_0(\mathfrak{g},M)$ is in the image of $VE_0$ if and only if \begin{equation}\label{van Est condition}
\int_{S^{n+1}_x}\omega \in Z\;\;\text{     for all } x \in G^0 \text{ and all } S^{n+1}_x\,,
\end{equation}
 where $S^{n+1}_x$ in an $(n+1)$-sphere contained in the source fiber over $x\,.$\footnote{For the case of smooth Lie groups, this seems to be shown in~\cite{wockel}, although our proof is still different.}
 \end{theorem}
\begin{proof}
The statement regarding $VE$ follows from the fact that 
\begin{equation}\label{iso of coho}
 \  H^*(\mathfrak{g},M)= H^*(\mathbf{E}^\bullet G,\kappa^{-1}\mathcal{O}(M)_\bullet)
\end{equation}
and Theorem~\ref{spectral theorem}. For the statement regarding $VE_0$ we use the fact that
\begin{align*}
    H^*_0(\mathfrak{g},M)\cong  H^{*+1}(\mathbf{E}^\bullet G,\widehat{\kappa^{-1}\mathcal{O}(M)}_\bullet)\,, \end{align*}
and the fact that the map
\begin{align*}
    H^{*+1}(\mathbf{E}^\bullet G,\kappa^{-1}\mathcal{O}(M)^0_\bullet)\xrightarrow{}H^{*+1}(\mathbf{E}^\bullet G,\widehat{\kappa^{-1}\mathcal{O}(M)}_\bullet) 
\end{align*}
is an isomorphism in degrees $\le n-1$ and is injective in degree $n\,.$
Furthermore, Theorem~\ref{spectral theorem} implies that \begin{align*}
    H^*(\mathbf{B}^{\bullet}G,\mathcal{O}(M)^0_\bullet)\xrightarrow{\kappa^{-1}}H^*(E\mathbf{G}^{\bullet},\kappa^{-1}\mathcal{O}(M)^0_\bullet)
    \end{align*}
is an isomorphism in degrees $\le n-1$ and is injective in degree $n\,,$ hence we get that the map
$H^*_0(G,M)\to H^*_0(\mathfrak{g},M)$ is an isomorphism in degrees $\le n-1$ and injective in degree $n\,.$ The statement regarding its image in degree $n$ follows from Corollary~\ref{groupoid kernel}.
\end{proof}
%\theoremstyle{definition}\begin{rmk}\label{van Est affine}
%Theorem~\ref{van Est image} holds in the holomorphic category as well.
%\end{rmk}
\theoremstyle{definition}\begin{exmp}This is a continuation of Example~\ref{module example}.
The source fibers of $\Pi_1(S^1)\rightrightarrows S^1$ are contractible, hence Theorem~\ref{van Est image} shows that
the cohomology groups are $H^i(\Pi_1(S^1),\tilde{\mathbb{R}})=0$ in all degrees, and 
\[
H^i(\Pi_1(S^1),\widetilde{\mathbb{R}/\mathbb{Z}})=H^{i+1}(\Pi_1(S^1),\tilde{\mathbb{Z}})=\begin{cases}
      \mathbb{Z}/2\mathbb{Z}, & \text{if}\ i= 0 \\
      0, & \text{if}\ i\ne 0\,,
    \end{cases}
\] and this result agrees with the computation done in Example~\ref{module example}.
We also have that $\Pi_1(S^1)$ is Morita equivalent to the fundamental group $\pi_1(S^1)\cong\mathbb{Z}\,,$ and the associated $\mathbb{Z}$-modules are the abelian groups $\mathbb{Z}\,,\,\mathbb{R}\,,\,\mathbb{R}/\mathbb{Z}\,,$ where even integers act trivially and odd integers act by inversion. One can also use this information to compute $H^{i}(\Pi_1(S^1),\tilde{\mathbb{Z}})$
and indeed find that \[
H^{i}(\Pi_1(S^1),\tilde{\mathbb{Z}})=\begin{cases}
      \mathbb{Z}/2\mathbb{Z}, & \text{if}\ i= 1 \\
      0, & \text{if}\ i\ne 1\,.
    \end{cases}
\]\end{exmp} 

\subsection{Groupoid Extensions and the van Est Map}
To every extension 
\begin{equation}\label{ab ext}
     1\to A\to E\to G\to 1
    \end{equation}
    of a Lie groupoid $G$ by an abelian group $A$ (see~\ref{abelian extensions}) one can associate a class in $H^1_0(\mathfrak{g},A)$ (where $\mathfrak{g}$ is the Lie algebroid of $G$) in two ways: one is given by the extension class of the short exact sequence $0\to \mathfrak{a}\to\mathfrak{e}\to\mathfrak{g}\to 0$ determined by~\ref{ab ext}, and the other is given by applying the van Est map to the class in $H^1_0(G,A)$ determined by~\ref{ab ext}. Here we will show that these two classes are the same.
\begin{theorem}\label{groupoid extension}
Let $M$ be a $G$-module and consider an extension of the form
\begin{align*}
    1\to M\to E\to G\to 1
\end{align*}
and let $\alpha\in H^1_0(G,M)$ be its isomorphism class. Then the isomorphism class of the Lie algebroid associated to $VE(\alpha)\in H^1_0(\mathfrak{g},M)$ is equal to the isomorphism class of the Lie algebroid $\mathfrak{e}$ of $E\,.$
\end{theorem}
\begin{proof}
Let $\{U_i\}_i$ be an open cover of $G^0\xhookrightarrow{} G$ on which there are local sections $\sigma_i:U_i\to E$ such that $\sigma$ takes $G^0\xhookrightarrow{}G$ to $G^0\xhookrightarrow{}E\,.$ These define a class $\alpha\in H^1_0(G,M)$ by taking $g_{ij}=\sigma^{-1}_i\cdot \sigma_j$ on $U_i\cap U_j\,,$ and where $h_{ijk}=\sigma_k^{-1}\cdot \sigma_i\cdot \sigma_j$ on $p_1^{-1}(U_i)\cap p_2^{-1}(U_j)\cap m^{-1}(U_k)\subset \mathbf{B}^{2}G\,.$ The sections $\sigma_{ii}$ induce a splitting of 
\begin{align*}
    0\to\mathfrak{m}\vert_{U_i}\to\mathfrak{e}\vert_{U_i}\to\mathfrak{g}\vert_{U_i}\to 0\,,
\end{align*}
which in turn gives a canonical closed 2-form $\omega\in C^2(\mathfrak{g}\vert_{U_i},M)\,,$ and the isomorphism given by $g_{ij}:E\vert_{U_i\cap U_j}\to E\vert_{U_i\cap U_j}$ induces an isomorphism  $\mathfrak{e}\vert_{U_i\cap U_j}\to \mathfrak{e}\vert_{U_i\cap U_j}$ given by $g_{ij*}$ (ie. the pushforward). Now the argument in Theorem 5 in~\cite{Crainic} implies that $VE(h_{iii})=[\omega_i]\,,$ and then one can check that $VE(\alpha)$ is the class given by $\{(\omega_i,g_{ij*})\}_{ij}\,.$
\end{proof}
\begin{comment}
$\mathbf{}$
\\In the holomorphic category we have the following:
\begin{theorem}\label{van Est image}
Suppose $G\rightrightarrows G^0$ is source $(n-1)$-connected, and that $M$ is a $G$-module. Then the van Est map  $\,VE:H^*(G,M)\to H^*(\mathfrak{g},M)$ is an isomorphism in degrees $\le n-1$ and is injective in degree $n\,.$  Furthermore, let $\omega\in H^n(\mathfrak{g},M)$ be such that there is an $\tilde{\omega}\in H^{n}(\mathfrak{g},\mathfrak{m})$ such that
\begin{align*}
    \tilde{\omega}\mapsto \omega 
    \end{align*}
    under the map 
    \begin{align*}
        H^{n}(\mathfrak{g},\mathfrak{m})\to H^n(\mathfrak{g},M)\,.
    \end{align*}
Then $\omega$ is in the image of $VE$ if and only if \begin{equation}\label{van Est condition}
\int_{S^n_x}\tilde{\omega} \in Z\;\;\text{     for all } x \in X \text{ and all } S^n_x\,,
\end{equation}
 where $S^n_x$ in an $n$-sphere contained in the source fiber over $x\,.$ 
\end{theorem}
\end{comment}
\section{Applications}
\subsection{Groupoid Extensions and Multiplicative Gerbes}
Here we describe applications of the main theorem (\Cref{van Est image}) to the integration of Lie algebroid extensions, to representations, and to multiplicative gerbes. 
\\\\If we take $M=E$ to be a representation in Theorem~\ref{van Est image}, then $Z=\{0\}$ and we obtain the following result, due to Crainic (see~\cite{Crainic}).
\begin{theorem}[\cite{Crainic}]\label{}
Suppose $G\rightrightarrows X$ is source $(n-1)$-connected and that $E$ is a $G$-representation. Then the van Est map  $\,VE:H^*(G,E)\to H^*(\mathfrak{g},E)$ is an isomorphism in degrees $\le n-1$ and is injective in degree $n\,.$  Furthermore, $\omega\in H^n(\mathfrak{g},E)$ is in the image of $VE$ if and only if \begin{equation}\label{}
\int_{S^n_x}\omega=0\;\;\text{     for all } x \in X \text{ and all } S^n_x\,,
\end{equation}
 where $S^n_x$ in an $n$-sphere contained in the source fiber over $x\,.$ 
\end{theorem}
Now we will prove a result about the integration of Lie algebroid extensions, which generalizes the above result in the $n=2$ case. At least in the case where $M=S^1$ this is due to Crainic and Zhu (see~\cite{zhu}), but there proof is different.
\begin{theorem}\label{central extension}
Consider the exponential sequence $0\to Z \to\mathfrak{m}\overset{\exp}{\to} M\,.$ Let
\begin{align}\label{LA extension}
    0\to\mathfrak{m}\to \mathfrak{a} \to \mathfrak{g}\to 0
\end{align}
be the central extension of $\mathfrak{g}$ associated to $\omega\in H^2(\mathfrak{g},\mathfrak{m})\,.$ Suppose that $\mathfrak{g}$ has a simply connected integration $G\rightrightarrows X$ and that 
\begin{align}\label{condition Z}
    \int_{S^2_x}\omega\in Z
\end{align}
for all $x\in X$ and $S^2_x\,,$ where $S^2_x$ in a $2$-sphere contained in the source fiber over $x\,.$  Then $\mathfrak{a}$ integrates to a unique extension
\begin{align}\label{LG extension}
1\to M\to A \to G\to 1\,.
\end{align}
\end{theorem}
In particular, if $G$ and $M$ are Hausdorff then $\mathfrak{a}$ admits a Hausdorff integration.\footnote{This generalizes a theorem proved by Crainic in~\cite{Crainic}, with a different proof.} 
\begin{proof}
By Theorem~\ref{van Est image} $H^1_0(G,M)$ is isomorphic to the subgroup of $H^1_0(\mathfrak{g},M)$ which have periods in $Z$ along the source fibers. Hence by Theorem~\ref{groupoid extension} the Lie algebroid extension in~\ref{LA extension} integrates to an extension of the form~\ref{LG extension}. Since in particular $\mathbf{B}^1A$ is a principal $M$-bundle over $G\,,$ it must be Hausdorff if $G$ and $M$ are.
\end{proof}
\begin{comment}\theoremstyle{definition}\begin{rmk}
Let us bring Theorem~\ref{central extension} into contact with the integrability results in~\cite{rui}. Consider an abelian extension of the form
\[
0\to \mathfrak{a}\to\mathfrak{g}\to TX\to 0\,.
\]
Lemma 3.6 and Theorem 4.1 in~\cite{rui}
imply that if we take the curvature class $\omega$ associated to a splitting of this extension, then $\mathfrak{g}$ is integrable if and only if the monodromy groups, given by
\[
N_x(\mathfrak{g})=\Big\{\int_{S^2_x}\omega\,:x\in X\,,S^2_x\in s^{-1}(x)\Big\}\,,
\]
satisfy the following two conditions: for all $x\in X\,,$
\begin{itemize}
    \item $N_x(\mathfrak{g})\subset A_x$ is discrete\,,
    \item $\liminf_{y\to x} d(\{0\},N_y(\mathfrak{g})-\{0\})>0\,,$
\end{itemize}
where $d(\cdot,\cdot)$ is some norm on $\mathfrak{g}\,.$ These two conditions automatically hold if $\omega$ satisfies condition~\eqref{condition Z}, so they are seen to be consistent with one another. However, Theorem~\ref{central extension} says in addition that the integration of $\mathfrak{g}$ will be an abelian extension of $\Pi_1(X)\,,$ in particular implying the integration is Hausdorff.
\end{rmk} 
\end{comment}
\begin{remark}Note that in fact a stronger result than the above theorem can be made. Suppose we have an extension 
\begin{align}\label{nonabelian}
0\to\mathfrak{m}\xrightarrow{\iota}\mathfrak{a}\xrightarrow{\pi} \mathfrak{g}\to 0\,,
\end{align}where now $\mathfrak{m}$ isn't assumed to be abelian, so that $M$ is a nonabelian module. However, suppose there is a splitting of~\eqref{nonabelian} such that the curvature $\omega$ takes values in the center of $\mathfrak{m}\,,$ denoted $Z(\mathfrak{m})$ (which we assume is a vector bundle). Then two things occur: First, let $\sigma:\mathfrak{g}\to\mathfrak{a}$ denote the splitting. Then we get an action of $\mathfrak{g}$ on $\mathfrak{m}$ defined by $
\iota(L_{X}W):=[\sigma(X),\iota(W)]\,,$ for $X\in \mathcal{O}(\mathfrak{g})\,,W\in \mathcal{O}(\mathfrak{m})$ (here we are defining $L_X W\,.$ One can check that this is in the image of $\iota$ and so defines a local section of $\mathcal{O}(\mathfrak{m})\,,$ and that this action is compatible with Lie brackets). Assume that this action integrates to an action of $G$ on $M\,,$ making $M$ into a (nonabelian) G-module. 
\\\\The second thing that occurs is that we get a central extension given by
\begin{align}\label{nonabelian2}
0\to Z(\mathfrak{m})\to (Z(\mathfrak{m})\oplus\mathfrak{g},\omega)\to \mathfrak{g}\to 0\,,
\end{align}
where $\omega$ is the curvature of $\sigma\,,$ and $\mathfrak{g}$ acts on $Z(\mathfrak{m})$ as above. The extension~\eqref{nonabelian2} is a reduction of~\eqref{nonabelian} in the following sense: we can form the Lie algebroid $Z(\mathfrak{m})\oplus \mathfrak{m}$ and this Lie algeboid has a natural action of $Z(\mathfrak{m})\,,$ and the quotient is isomorphic to $\mathfrak{m}\,.$ Similarly, we can form the Lie algebroid $(Z(\mathfrak{m})\oplus\mathfrak{g},\omega)\oplus \mathfrak{m}\,,$ and this Lie algebroid also has a natural action of $Z(\mathfrak{m})\,,$ and the quotient is isomorphic to $\mathfrak{a}\,.$ Therefore, the extension~\eqref{nonabelian} is associated to the extension~\eqref{nonabelian2} in a way that is analogous to the reduction of the structure group of a principal bundle.
\\\\Assume now that the extension~\eqref{nonabelian2} integrates to an extension 
\begin{align}\label{abelian ext}
1\to Z(M)\xrightarrow{\iota} E \xrightarrow{\pi}  G\to 1\,,
\end{align}
where $G$ is the source simpy connected groupoid integrating $\mathfrak{g}\,.$ Then we can form the product Lie groupoid $E _{s}{\times}_{s} M:$ the multiplication is given by \begin{align*}
(e,m)(e',m')=(ee',m(\pi(e)^{-1}\cdot m'))\,,
\end{align*} where $t(e)=s(e')\,.$ Similarly to the Lie algebroid extension case, the family of abelian groups $Z(M)$ acts on the family of groups $Z(M)_{s}{\mathop{\times}}_s M\,,$ as well as on the Lie groupoid $E _{s}{\mathop{\times}}_s M\,,$ and the quotient of the former is isomorphic to $M\,,$ and the quotient of the latter integrates $\mathfrak{a}$ in~\eqref{nonabelian}. This gives us an extension
\begin{align}
    1\to M\to A\to G\to 1
\end{align}
integrating~\eqref{nonabelian}. Therefore, if  we can integrate~\eqref{nonabelian2} we can also integrate~\eqref{nonabelian}.
\\\\One should notice the similarity between the construction we've just described and the construction described in Lemma 3.6 in~\cite{rui}, in the special case of a regular Lie algebroid (ie. where the anchor map has constant rank), and where the extension is given by
\begin{align}
    0\to \text{ker}(\alpha)\to \mathfrak{a}\xrightarrow{\alpha}\text{im}(\alpha)\to 0\,,
\end{align}
where $\alpha$ is the anchor map of $\mathfrak{a}\,.$ The obstruction to integration described there coincides with the obstruction given by~\Cref{van Est image} for the integration of~\eqref{nonabelian2}, and we've shown that the vanishing of this obstruction is sufficient for the integration of~\eqref{nonabelian}, and hence of $\mathfrak{a}\,,$ to exist. 
\begin{comment}So the argument for the integration of Lie algebroids presented in~\cite{rui} is a kind of nonabelian and singular analogue of the argument for the integration of abelian extensions given by~\Cref{van Est image} (singular because if the Lie algebroid isn't regular its kernel and image aren't vector bundles).
\end{comment}
\end{remark}
\begin{comment}\begin{theorem}
Suppose $(X,\Lambda)$ is a regular Poisson manifold and that the kernel of $\Lambda$ has rank $n\,.$. Let $\mathcal{F}$ denote the foliation groupoid associated to the foliation of $X$ by symplectic leaves. Suppose further that there exists an $\mathcal{F}$-module $M$ with Lie algebroid $\mathbb{R}^n_X$ such that, for each symplectic leaf $L\,,$ the subgroup generated by integrating $\Lambda$ over all classes in $\pi_2(L)$ lies in the kernel of $\exp:\mathbb{R}^n_X\to M\,.$  
\end{theorem}
\end{comment}
The above results concerned the degree $1$ case in truncated cohomology. We will now apply the main theorem to the integration of rank one representations, which concerns degree $1$ in nontruncated cohomology. First we make use of the following result:
\begin{proposition}\label{G representations}
The group of isomorphism classes of representations of $G\rightrightarrows G^0$ on complex line bundles is isomorphic to $H^1(G\,,\mathbb{C}^*_{G^0})\,.$ The corresponding statement for real line bundles holds, with $\mathbb{C}^*_{G^0}$ replaced by $\mathbb{R}^*_{G^0}\,.$ See Example~\ref{representation}.
\end{proposition}
The following statement is already known, we are just giving a cohomological proof.
\begin{theorem}
Let $G\rightrightarrows G^0$ be a source simply connected Lie groupoid. Then $\text{Rep}(G,1)\cong \text{Rep}(\mathfrak{g},1)\,,$ where
$\text{Rep}(G,1)\,,\,\text{Rep}(\mathfrak{g},1)$ are the categories of 1-dimensional representations, ie. representations on line bundles.
\end{theorem}
\begin{proof}
This follows directly from Theorem~\ref{van Est image}, Example~\ref{representation} and Proposition~\ref{G representations}.
\end{proof}
\begin{comment}In the same vein as the previous theorem, we have the following:
\begin{theorem}
Let $G\rightrightarrows G^0$ be a source simply connected Lie groupoid. Then every principal $M$-bundle with a flat connection integrates uniquely, up to isomorphism, to a nonlinear representation of $G$ (see the appendix for the definition of nonlinear representation).
\end{theorem}
\begin{proof}
This follows directly from Theorem~\ref{van Est image}.
\end{proof}
$\mathbf{}$
\end{comment}
Now for the degree $2$ case in truncated cohomology: we use the main theorem to give a proof of an integration result concerning the multiplicative gerbe on compact, simple and simply connected Lie groups (see~\cite{konrad}). 
\begin{theorem}
Let $G$ be a simply connected Lie group. Then for each $\alpha\in H^2_0(\mathfrak{g},\mathbb{R})$ which is integral on $G\,,$ there is a class in $H^2_0(G,S^1)$ integrating it.
\end{theorem}
\begin{proof}
It is well known that simply connected Lie groups are 2-connected, so Theorem~\ref{van Est image} immediately gives the result.
\end{proof}
\subsection{Group Actions and Lifting Problems}
In this section we apply~\Cref{van Est image} to study the problems of lifting projective representations to representations, and to lifting Lie group actions to principal torus bundles.
\subsubsection{Lifting Projective Representations}
\begin{theorem}
Let $G$ be a simply connected Lie group and let $V$ be a finite dimensional complex vector space. Let $\rho:G\to \text{PGL } (V)$ be a homomorphism. Then $G$ lifts to a homomorphism $\tilde{\rho}:G\to\text{GL } (V)\,.$ If $G$ is semisimple, this lift is unique.
\end{theorem} 
\begin{proof}
We have a central extension
\begin{align}\label{GL extension}
    1\to \mathbb{C}^*\to \text{GL } (V)\to \text{PGL } (V)\to 1\,,
\end{align}
and the corresponding Lie algebra extension splits: the Lie algebra of $\text{PGL } (V)$ is isomorphic to $\mathfrak{g}\mathfrak{l}(V)/\mathbb{C}\,,$ where $\lambda\in\mathbb{C}$ acts on $X\in\mathfrak{g}\mathfrak{l}(V)$ by taking $X\mapsto X+\lambda\,\mathbf{I}\,.$ The map \begin{align*}
    \mathfrak{g}\mathfrak{l}(V)/\mathbb{C}\to\mathfrak{g}\mathfrak{l}(V)\,, X\mapsto X-\frac{\text{tr}(X)}{\text{dim}(V)}\mathbf{I}
    \end{align*}
    is a Lie algebra homomorphism. Therefore, since $G$ is simply connected,~\Cref{van Est image} implies that the extension of $G$ that we get by pulling back the extension given by \eqref{GL extension} via $\rho$ is trivial (since the pullback of a trivial Lie algebra extension is trivial). However, a trivialization of the pullback extension is the same thing as a lifting of the homomorphism $\rho$ to a homomorphism $\tilde{\rho}:G\to\text{GL } (V)\,,$ hence such a lifting exists. 
    \\\\Now for uniqueness: it is easy to see that the liftings of $\rho$ are a torsor for $\text{Hom}(G,\mathbb{C}^*)\,,$ but again by~\Cref{van Est image} we have that $\text{Hom}(G,\mathbb{C}^*)\cong \text{Hom}(\mathfrak{g},\mathbb{C})\,,$ and the right side is $0$ if $G$ is semisimple. Hence if $G$ is semisimple there is a unique lift.
    
\end{proof}
\begin{remark}
One can also use the above method to give a proof of Bargmann's theorem, that is, if $H^2(\mathfrak{g},\mathbb{R})=0\,,$ then every projective representation of a (infinite dimensional) Hilbert space lifts to a representation.
\end{remark}
\subsubsection{Lifting Group Actions to Principal Bundles}
Now we will look at a different lifting  problem, one involving compact, semisimple Lie groups. First let us remark the following well-known result:
\begin{lemma}
A compact Lie group is semisimple if and only if its fundamental group is finite.
\end{lemma}
Now the aim of the rest of this section is to prove the following result:
\begin{theorem}\label{compact group actions}
Let $G$ be a compact, semisimple Lie group acting on a manifold $X\,.$ Suppose $P\to X$ is a principal bundle for the $n$-torus $T^n\,.$ Then the action of $G$ on $X$ lifts to an action of $G$ on $P^{|\pi_1(G)|}$ (here $P^{|\pi_1(G)|}$ is the principal $T^n$-bundle whose torsor over $x\in X$ is the product of the torsor over $x$ in $P$ with itself $|\pi_1(G)|$ times), and the lift is unique up to isomorphism (ie. any two lifts differ by a principal bundle automorphism). 
\end{theorem}
In particular, if $G$ is compact and simply connected, then actions of $G$ on a manifold $X$ lift to all principal $T^n$-bundles over $X\,.$
\theoremstyle{definition}\begin{exmp}
Consider the standard action of $SO(3)$ on $S^2\,.$ We have that $\pi_1(SO(3))=\mathbb{Z}/2\mathbb{Z}\,,$ hence $|\pi_1(SO(3))|=2\,.$ Therefore,~\Cref{compact group actions} implies that the action of $SO(3)$ on $S^2$ lifts to an action on all even degree prinicpal $S^1$-bundles over $S^2\,,$ in a unique way up to isomorphism. On the other hand, since $SU(2)$ is simply connected, its standard action on $S^2$ lifts to an action on all principal $S^1$-bundles over $S^2\,,$ again in a unique way up to isomorphism.
\end{exmp}
Before proving~\Cref{compact group actions}, we will prove the following result, which is interesting in its own right and is related to the Riemann-Hilbert correspondence\footnote{In particular, this result determines exactly when a flat connection on a gerbe integrates to an action of the fundamental groupoid on the gerbe.}
\begin{lemma}\label{fundamental group}
Let $X$ be a connected manifold with universal cover $\tilde{X}$ and suppose that $\pi_k(X)=0$ for all $2\le k\le m\,.$ Let $T^n\xrightarrow{\text{dlog}} \mathbf{\Omega}^\bullet$ be the Deligne complex. Then $H^k(\pi_1(X),T^n)\cong H^k(X,T^n\xrightarrow{\text{dlog}} \mathbf{\Omega}^\bullet)$ for all $2\le k\le m\,,$ and the following sequence is exact: \begin{align}
    0\to H^{m+1}(\pi_1(X),T^n)\to H^{m+1}(X,T^n\to \mathbf{\Omega}^\bullet)\to H^{m+1}(\tilde{X},T^n\to \mathbf{\Omega}^\bullet)\,.
\end{align}
\end{lemma}
\begin{proof}
We have that $\pi_1(X)$ is Morita equivalent to $\Pi_1(X)\,,$ so by Morita invariance \begin{align*}
H^{\bullet}(\pi_1(X),T^n)\cong H^{\bullet}(\Pi_1(X),T_X^n)\,.\end{align*}
The result then follows from~\Cref{van Est image}.
\end{proof}
\begin{corollary}\label{power is zero}
Let $G$ be a connected Lie group and as usual let $\mathbf{B}^1G$ be the underlying manifold. Then for every class $\alpha \in H^1(\mathbf{B}^1G,T^n\to \mathbf{\Omega}^\bullet)\,,$ we have that $\alpha^{|\pi_1(G)|}=1\,.$ 
\end{corollary}
\begin{proof}
From~\Cref{fundamental group} we have the well-known result that $H^1(\mathbf{B}^1G,T^n\to \mathbf{\Omega}^\bullet)\cong H^1(\pi_1(G),T^n)\,.$ The latter is equal to $\text{Hom}(\pi_1(G),T^n)\,,$ however every $f\in \text{Hom}(\pi_1(G),T^n)$ satisfies $f^{|\pi_1(G)|}=1\,,$ completing the proof.
\end{proof}
We now state a proposition that will be needed for the proof of~\Cref{compact group actions} (for a proof of this proposition, see~\cite{Crainic}).
\begin{proposition}\label{vanishing cohomology prop}
Let $G\rightrightarrows X$ be a proper Lie groupoid (ie. the map $(s,t):G\to X\times X$ is a proper map). Let $E\to X$ be a representation of $G\,.$ Then $H^k(G,E)=0$ for all $k\ge 1\,.$ 
\end{proposition}
The key to proving~\Cref{compact group actions} is the following lemma:
\begin{lemma}\label{vanishing cohomology}
Let $G$ be a compact, simply connected Lie group acting on a manifold $X\,.$ Then $H^1_0(G\ltimes X,T^n_X)=0\,.$
\end{lemma}
\begin{proof}
Since $G$ is compact the action is proper, hence $G\ltimes X$ is a proper groupoid, hence from~\Cref{vanishing cohomology prop} we see that $H^k(G\ltimes X,\mathbb{R}^n_X)=0$ for all $k\ge 1\,.$ This implies that $H^k_0(G\ltimes X,\mathbb{R}^n_X)=0$ for all $k\ge 2\,.$  Since simply connected Lie groups are $2$-connected,~\Cref{van Est image} implies that $H^2_0(G\ltimes X,\mathbb{Z}^n_X)=0\,.$ Hence, from the short exact sequence
$0\to\mathbb{Z}^n\to\mathbb{R}^n\to T^n\to 0\,,$ we get that $H^1_0(G\ltimes X,T^n_X)=0\,.$ 
\end{proof}
We are now ready to prove~\Cref{compact group actions} for simply connected groups.
\begin{lemma}\label{compact simply connected}
Let $G$ be a compact, simply connected Lie group acting on a manifold $X\,.$ Suppose $P\to X$ is a principal bundle for the $n$-torus $T^n\,.$ Then the action of $G$ on $X$ lifts to an action of $G$ on $P\,,$ and the lift is unique up to isomorphism.
\end{lemma}
\begin{proof}
Consider the gauge groupoid of $P$ given by $\text{At}(P):=P\times P/T^n\rightrightarrows X\,,$ where the action of $T^n$ is the diagonal action (here the source and target maps are the projections onto the first and second factors, respectively, and a morphism with source $x$ and target $y$ is a $T^n$-equivariant morphism between the fibers of $P$ lying over $x$ and $y\,,$ respectively). The gauge groupoid fits into a central extension of $\text{Pair}(X)\,,$ ie.
\begin{align}\label{Atiyah sequence}
    1\to T^n_X\to \text{At}(P)\to\text{Pair}(X)\to 1\,.
\end{align}
A lift of the $G$-action to $P\to X$ is equivalent to a lift of the canonical homomorphism $G\ltimes X\xrightarrow{(s,t)}\text{Pair}(X)$ to $\text{At}(P)\,,$ which is equivalent to a trivialization of the central extension of $G\ltimes X$ given by pulling back, via $(s,t)\,,$ the central extension given by \eqref{Atiyah sequence}. From~\Cref{vanishing cohomology} we know that such a trivilization exists, hence the $G$-action lifts to $P\,.$
\\\\Uniqueness up to isomorphism follows from the fact that the isomorphism classes of different lifts are a torsor for the image of $H^0_0(G\ltimes X,T^n)$ in $H^1(G\ltimes X,T^n)\,,$ and that the image is trivial follows from the exponential sequence $1\to \mathbb{Z}^n\to\mathbb{R}^n\to T^n\to 1\,,$ since both $H^1(G\ltimes X,\mathbb{R}^n_X)$ and $H^1_0(G,\mathbb{Z}^n_X)$ are trivial (the former follows from~\Cref{vanishing cohomology prop}, the latter follows from~\Cref{van Est image}).
\end{proof}
Now we can prove~\Cref{compact group actions}. One way of doing this is to look at the action of $\pi_1(G)$ on its universal cover, another way is  the following:
\begin{proof}[Proof of Theorem \ref{compact group actions}]
Let $\tilde{G}$ be the universal cover of $G\,.$ From~\Cref{compact simply connected} we know that the corresponding action of $\tilde{G}$ on $X$ lifts to an action on $P\,,$ giving us a class $\alpha\in H^1(\tilde{G}\ltimes X,T^n)$ whose underlying principal bundle on $X$ is $P\,.$ Hence after applying the van Est map we get a class $\text{VE}(\alpha)\in H^1(\mathfrak{g}\ltimes X,T^n)\,,$ whose underlying principal bundle on $X$ is also $P\,.$ 
\\\\After translating $\text{VE}(\alpha)$, we get a flat $T^n$-bundle on each source fiber of $G\ltimes X\,,$ ie. for each $x\in X$ we get a flat $T^n$-bundle on $G$, which we denote by $P_{(G,x)}\,.$ Then by~\Cref{power is zero} we have that $P_{(G,x)}^{|\pi_1(G)|}$ is trivial. However, $P_{(G,x)}^{|\pi_1(G)|}$ is the right translation of $|\pi_1(G)|\cdot\text{VE}(\alpha)$ (where the $\mathbb{Z}$-action is the natural one on cohomology classes),
hence by~\Cref{van Est image} we get the existence of a lift of the $G$-action to $P^{|\pi_1(G)|}\,.$ 
\\\\Uniqueness follows from the same argument as in~\Cref{compact simply connected}.
 \end{proof}
\begin{comment}
\begin{proof}
Because Lie groups have vanishing second homotopy group,~\Cref{fundamental group} implies that $H^2(G^{(1)},\mathbb{Z})\cong H^2(\pi_1(G),\mathbb{Z})\,.$ Since $|\pi_1(G)|$ is finite, every class in $H^2(\pi_1(G),\mathbb{Z})\,,$ and hence every class in $H^2(G^{(1)},\mathbb{Z})\,,$ is torsion. So we have that $H^1(G^{(1)},\mathbb{C})=H^1(\pi_1(G),\mathbb{C})=0\,,$ and  $H^2(G^{(1)},\mathbb{C})=H^2(\pi_1(G),\mathbb{C})=0\,,$ 
Hence from the short exact sequence \begin{align*}
    0\to \mathbb{Z}\to\mathbb{C}\to\mathbb{C}^*\to 0
\end{align*}
we get that 
$H^1(G^{(1)},\mathbb{C}^*)\cong H^1(\pi_1(G),\mathbb{C}^*)\,,$ which implies the result since for every $\gamma\in\pi_1(G)$ we have that $\gamma^{|\pi_1(G)|}=1\,.$ 
 \end{proof}
\end{comment}
\subsection{Quantization of Courant Algebroids}
In this section we will discuss applications of our main theorem to the quantization of Courant algebroids, as discussed in~\cite{gen kahler}. 
\\\\Let $C$ be a smooth Courant algebroid over $X$ associated to a $3$-form $\omega\,,$ and suppose that it is prequantizable, that is $\omega$ has integral periods. Let $g$ denote an $S^1$-gerbe prequantizing $\omega\,.$ Let $D\subset C$ be a Dirac structure. Then in particular, $D$ is a Lie algebroid, and as explained in~\cite{gen kahler} $g$ can be equipped with a flat $D$-connection, denoted $A\,.$ This determines a class $[(g,A)]\in H^2(D,S^1_X)\,.$ Suppose $D$ integrates to a Lie groupoid. We can then ask about the integrability of $[(g,A)]\,,$ or in other words: does the action of $D$ on $g$ integrate to an action of the corresponding source simpy connected groupoid on $g\,?$
\begin{comment}\\\\First we will give an example for which it doesn’t.
\theoremstyle{definition}\begin{exmp}Let $X=S^2$ and consider the trivial Courant bracket on $TS^2\oplus T^*S^2\,.$ Let $\omega$ be a $2$-form which doesn’t have integral periods, and let $D$ be the graph of $\omega: TS^2\to T^*S^2\,.$ This Courant algebroid prequantizes to the trivial gerbe with the flat connection given by $\omega\,.$ Now $\Pi_1(S^2)=\text{Pair}\,(S^2)\,,$ and since $\omega$ doesn't have integral periods, by Theorem~\ref{van Est image} the gerbe with flat connection doesn’t integrate.
\end{exmp}
\end{comment}
Here we give a class of examples that does integrate, and it relates to the basic gerbe on a compact, simple Lie group. $($see~\cite{erbe},~\cite{Severa}$)\,.$
\theoremstyle{definition}\begin{exmp}
Let $G$ be a compact, simple Lie group with universal cover $\tilde{G}\,.$ 
and let $\langle \cdot,\cdot\rangle$ be the unique bi-invariant $2$-form which at the identity is equal to the Killing form. Associated to $\langle \cdot,\cdot\rangle$ is a bi-invariant and integral $3$-form $\omega$, called the Cartan $3$-form, given at the identity by
\begin{align*}\omega\vert_e=\frac{\langle \,[\,\cdot\,,\,\cdot\,]\,,\cdot\rangle\vert_e}{2}\,.
\end{align*}
\begin{comment}
From this we get a Courant algebroid with Dirac structure given by 
\begin{align*}
C_D\vert_g=\{\big(L_{g*}a-R_{g*}a, \langle L_{g*}a+R_{g*}a,\cdot\rangle/2\big):a\in\mathfrak{g}\}\,,
\end{align*} 
where $L_g\,,R_g$ denote left and right multiplication by $g\,,$ respectively.
\end{comment}
The Dirac structure in this case, called the Cartan-Dirac structure, is the action Lie algebroid $\mathfrak{g}\ltimes G\,,$ where the action is the adjoint action of $\mathfrak{g}$ on $G\,.$
%From this we get a gerbe, called the basic gerbe, with a flat $C_D$ connection. Furthermore 
From this there is a canonical class $\alpha\in H^2(\mathfrak{g}\ltimes G,S^1_G)\,,$ whose underlying gerbe on $G$ is called the basic gerbe. The source simply connected integation of $\mathfrak{g}\ltimes G$ is $\tilde{G}\ltimes G\,,$ where the action of $\tilde{G}$ on $G$ is the one lifting the action of $G$ on itself by conjugation. Since the source fibers of $\tilde{G}\ltimes G$ are diffeomorphic to $\tilde{G}\,,$ which is necessarily $2$-connected, by Theorem~\ref{van Est image} we have that 
\begin{align*}
    H^2(\tilde{G}\ltimes G,S^1_G)\overset{VE}{\cong}H^2(\mathfrak{g}\ltimes G,S^1_G)\,.
\end{align*} Hence $\alpha$ integrates to a class in $H^2(\tilde{G}\ltimes G,S^1_G)\,.$
\end{exmp}To summarize, we have proven the following [see~\cite{erbe},~\cite{krepski}]:
\begin{theorem}[Integration of Cartan-Dirac structures]
Let $G$ be a compact, simple Lie group with universal cover $\tilde{G}\,.$ Then the adjoint action of $\mathfrak{g}$ on the basic gerbe (where the action is given by the Cartan-Dirac structure) integrates to an action of $\tilde{G}$ on the basic gerbe.
\end{theorem} 
\subsection{Integration of Lie $\infty$-Algebroids}
In this section we will discuss the integration and quantization of Lie $\infty$-algebroids. See~\cite{pym2} for more details. We consider Lie $\infty$-algebroids of the following form: 
\\\\Let $\mathfrak{g}$ be a Lie algebroid and let $M\to X$ be a $\mathfrak{g}$-module. Let $\omega\in C^n(\mathfrak{g},M)$ be closed, $n>2\,.$ We can define a two term Lie $(n-1)$-algebroid as follows: Let $\mathcal{L}=\mathfrak{m}\oplus\mathfrak{g}$ where $\mathfrak{m}$ has degree $2-n$ and $\mathfrak{g}$ has degree $0\,.$ Let all differentials be zero except for the degree $0$ and degree $-n$ differentials. Define the degree 
$0$ differential as follows: for $U$ an open set in $X$ and for $m_1\,,m_2\in\mathcal{O}(\mathfrak{m})(U)\,,g_1\,,g_2\in\mathcal{O}(\mathfrak{g})(U)\,,$ let
\begin{align*}
    [m_1+g_1,m_2+g_2]_0=[g_1,g_2]+d_{CE}m_2(g_1)-d_{CE}m_1(g_2)\,,
    \end{align*}
    where $[g_1,g_2]$ is the Lie bracket of $g_1\,,g_2$ in $\mathfrak{g}\,.$
    Define the degree $2-n$ bracket by as follows: for $g_1\,,\ldots\,,g_n\in\mathcal{O}(\mathfrak{g})(U)\,,$ let 
    \begin{align*}
        [g_1\,,\ldots\,,g_n]_n=\omega(g_1\,,\ldots\,,g_n)\,,
    \end{align*}
    otherwise if any of inputs is in $\mathcal{O}(\mathfrak{m})(U)$ let the bracket be zero. This defines a Lie $(n-1)$-algebroid. 
\\\\Since the universal cover of a $k$-dimensional torus $($for $k\ge 1)$ is contractible,~\Cref{van Est image} gives us the following result, at the level of cohomology:
\begin{corollary}
All Lie $(n-1)$-algebroids associated to closed $n$-forms on the $k$-dimensional torus $T^k$ integrate to multiplicative $(n-2)$-gerbes.
\end{corollary}
%\subsubsection{Lie $2$-Algebras}
We now apply the previous results to Lie $2$-algebras. As proved in~\cite{baez2}, all Lie $2$-algebras are equivalent to ones of the form
\begin{align}\label{lie 2}
V\to \mathfrak{g}\,,
\end{align}
where the only nonzero brackets are the degree $0$ and $-1$ brackets, and where the degree $-1$ bracket is given by a closed $3$-form. Furthermore, if $\omega\,,\omega'$ define equivalent Lie $2$-algebras, then $[\omega]=[\omega']$ in $H^3_0(\mathfrak{g},V)\,,$ implying that the map from Lie $2$-algebras to $H^3_0(\mathfrak{g},V)$ is canonical. Since simply connected Lie groups are $2$-connected, Theorem~\ref{van Est image} can help us determine when a Lie $2$-algebra integrates.
\begin{theorem}\label{2 algebra}
Let $\mathcal{L}$ be a Lie $2$-algebra represented by the $3$-form $\omega\,.$ Let $G$ be the simply connected integration of $\mathfrak{g}\,.$ Then if the periods $P(\omega)$ of $\omega$ form a discrete subgroup of $V\,,$ then $\mathcal{L}$ integrates to a class in $H_0^2(G,V/P(\omega))\,.$ 
\end{theorem}
\theoremstyle{definition}\begin{rmk}
Note that in~\cite{andre} it is shown that the obstruction to integrating a Lie $2$-algebra to a Lie $2$-group is that the periods of $\omega$ form a discrete subgroup of $V\,,$ ie. the obstruction is the same as the one in the above theorem. To explain this, we note the following: it is shown in~\cite{schommer} that to every class in $H^2_0(G,S^1)$ there corresponds an equivalence class of Lie $2$-groups. We expect that under this correspondence,~\Cref{2 algebra} shows that the Lie $2$-algebras which satisfy the hypotheses of this theorem integrate to Lie $2$-groups.
\end{rmk}
\subsection{van Est Map: Heisenberg Action Groupoids}\label{heisenberg action}
\begin{comment}In this section we will provide applications of the tools developed in the previous sections to an explicit example. That is,
\end{comment}
In this section we will apply the tools developed in the previous sections to integrate a particular Lie algebroid extension and show that we get a Heisenberg action groupoid.
\\\\Consider the space $\mathbb{C}^2$ with divisor $D=\{xy=0\}\,.$ Then the $2$-form \begin{align*}
    \omega=\frac{dx\wedge dy}{xy}
    \end{align*}
    is a closed form in $C_0^2(T_{\mathbb{C}^2}(-\log{D}),\mathbb{C}_{\mathbb{C}^2})\begin{footnote}{On $X\backslash D$ the $2$-form $\omega/2\pi i$ is the curvature of the Deligne line bundle associated to the holomorphic functions $x$ and $y\,.$}\end{footnote}\,.$ The source simply connected integration of $T_{\mathbb{C}^2}(-\log{D})$ is 
$\mathbb{C}^2\ltimes \mathbb{C}^2\,,$ where the action of $\mathbb{C}^2$ on itself is given by
\begin{align*}
    (a,b)\cdot (x,y)=(e^ax,e^by)\,.
\end{align*}
Since the source fibers are contractible Theorem~\ref{van Est image} tells us that the central extension of $T_{\mathbb{C}^2}(-\log{D})$ defined by $\omega$ integrates to an $\mathbb{C}_{\mathbb{C}^2}$ central extension of $\mathbb{C}^2\ltimes \mathbb{C}^2\,.$ We will describe the central extension here. First we will compute the integration of $\omega:$ we define coordinates on $\mathbf{B}^{\bullet\le 2} (\mathbb{C}{\mathop{\times}}\mathbb{C}\ltimes \mathbb{C}{\mathop{\times}}\mathbb{C})$ as follows:
\begin{align*}
&(x,y)\in \mathbb{C}^2= \mathbf{B}^0 (\mathbb{C}{\mathop{\times}}\mathbb{C}\ltimes \mathbb{C}{\mathop{\times}}\mathbb{C})\,,
\\&(a,b,x,y)\in \mathbb{C}^2{\mathop{\times}}\mathbb{C}^2=\mathbf{B}^1 (\mathbb{C}{\mathop{\times}}\mathbb{C}\ltimes \mathbb{C}{\mathop{\times}}\mathbb{C})\,,
\\& (a',b',a,b,x,y)\in \mathbf{B}^2 (\mathbb{C}{\mathop{\times}}\mathbb{C}\ltimes \mathbb{C}{\mathop{\times}}\mathbb{C})\,.
\end{align*}
On $\mathbf{E}^{\bullet\le 2} (\mathbb{C}{\mathop{\times}}\mathbb{C}\ltimes \mathbb{C}{\mathop{\times}}\mathbb{C})$ we have coordinates
\begin{align*}
&(a,b,x,y)\in \mathbf{E}^0 (\mathbb{C}{\mathop{\times}}\mathbb{C}\ltimes \mathbb{C}{\mathop{\times}}\mathbb{C})\,,
\\&(a',b',a,b,x,y)\in \mathbf{E}^1 (\mathbb{C}{\mathop{\times}}\mathbb{C}\ltimes \mathbb{C}{\mathop{\times}}\mathbb{C})\,,
\\& (a'',b'',a',b',a,b,x,y)\in \mathbf{E}^2 (\mathbb{C}{\mathop{\times}}\mathbb{C}\ltimes \mathbb{C}{\mathop{\times}}\mathbb{C})\,,
\end{align*}
where the map $\kappa:\mathbf{E}^{\bullet\le 2} (\mathbb{C}{\mathop{\times}}\mathbb{C}\ltimes \mathbb{C}{\mathop{\times}}\mathbb{C})\to\mathbf{B}^{\bullet\le 2} (\mathbb{C}{\mathop{\times}}\mathbb{C}\ltimes \mathbb{C}{\mathop{\times}}\mathbb{C})$
is given by
\begin{align*}
    &(a,b,x,y)\mapsto (e^a x,e^b y)\,,
    \\& (a',b',a,b,x,y)\mapsto (a',b',e^a x, e^b y)\,,
    \\& (a'',b'',a',b',a,b,x,y)\mapsto (a'',b'',a',b',e^a x, e^b y)\,.
\end{align*}
When we right translate $\omega$ to $\mathbf{E}^0(\mathbb{C}\ltimes\mathbb{C})$ we get the fiberwise form $da\wedge db\,.$ This is exact, with primitive $a\,db\,.$ When we pullback $a\,db$ to $\mathbf{E}^1(\mathbb{C}\ltimes\mathbb{C})$ we get the fiberwise form $a'\,db\,,$ and this is exact, with primitive $a'b\,.$ When we pullback $a'b$ to $\mathbf{E}^2(\mathbb{C}\ltimes\mathbb{C})$ we get the function $a''b'\,,$ and this is $\kappa^*a'b\,.$ So the cocycle integrating $\omega$ is $f(a',b',a,b,x,y)=a'b\,.$ 
\\\\One can show that the central extension associated to this cocycle is an action groupoid of the complex Heisenberg group acting on $\mathbb{C}{\mathop{\times}}\mathbb{C}\,,$ ie. we have the following proposition:
\begin{proposition}
The logarithmic $2$-form $\frac{dx\wedge dy}{xy}$ on $\mathbb{C}^2$ with divisor $xy=0$ defines a Lie algebroid extention of $T_{\mathbb{C}^2}(-\log{\{xy=0\}})\,.$ This Lie algebroid extension integrates to an extension of $\mathbb{C}^2\ltimes \mathbb{C}^2$ given by a Heisenberg action groupoid. More precisely,
the extension is of the form
\begin{align}\label{heisenberg extension}
    0\to\mathbb{C}_{\mathbb{C}^2}\to H\ltimes\mathbb{C}^2 \to\mathbb{C}^2\ltimes \mathbb{C}^2\to 0\,,
\end{align}
where $H$ is
the subgroup of matrices of the form
\begin{align*}
    \begin{pmatrix}
    1 & a & c\\
    0 & 1 & b \\
    0 & 0 & 1
    \end{pmatrix}
\end{align*}
for $a\,,b\,,c\in\mathbb{C},$ and the action on $\mathbb{C}{\mathop{\times}}\mathbb{C}$ is given by $(a,b,c)\cdot (x,y)=(e^a x,e^b y)\,,$ where $(a,b,c)$ represents the above matrix.
\end{proposition}

\section{The Canonical Module Associated to a Complex Manifold and Divisor}
Given a complex manifold $X$ and a (simple normal crossings) divisor $D\,,$ we construct a natural module for the Lie groupoid $\text{Pair}(X,D)$ (which is the terminal integration of $T_X(-\log{D})\,,$ the Lie algebroid whose sheaf of sections is the sheaf of sections of $T_X$ which are tangent to $D)\,.$ These are modules for which the underlying surjective submersion does not define a fiber bundle, and in particular the underlying family of abelian groups is not locally trivial. Generically the fiber will be $\mathbb{C}^*\,,$ but over $D$ the fibers will degenerate to $\mathbb{C}^*{\mathop{\times}} \mathbb{Z}^k\,,$ for some $k$ depending on the point $D\,.$
\subsection{The Module $\mathbb{C}^*_{\mathbb{C}}(*\{0\})$}
Here we will do a warm up example for the general case to come in the next section. More precisely, we will construct
a family of abelian groups whose sheaf of sections is isomorphic to the sheaf of nonvanishing meromorphic functions with a possible pole or zero
only at the origin, and we will show that it is naturally a module for the terminal groupoid integrating $T_\mathbb{C}(-\log{\{0\}})\,,$ the Lie algebroid whose sheaf of sections is
isomorphic to the sheaf of sections of $T\mathbb{C}$ vanishing at the origin. This space was defined in~\cite{luk}.
\\\\Consider the action groupoid $\mathbb{C}^*\ltimes\mathbb{C}\rightrightarrows\mathbb{C}\,,$ where the action of $\mathbb{C}^*$ on $\mathbb{C}$ is given by
\begin{align*}
    a\cdot x=ax\,.
\end{align*}This is the terminal groupoid integrating $T_\mathbb{C}(-\log{\{0\}})\,.$ We will construct a module for this groupoid as follows:
consider the family of abelian groups given by
\begin{align*}
{\mathbb{C}}{\mathop{\times}}\mathbb{C}^*{\mathop{\times}}\mathbb{Z}\overset{p_1}{\to}\mathbb{C}\,.
\end{align*}
This family of abelian groups is a $\mathbb{C}^*\ltimes\mathbb{C}$-module with action given by
\begin{align}
(a,x)\cdot(x,y,i)=(ax,a^{-i}y,i)\,.
\end{align}
There is a submodule given by
\begin{align*}
\mathbb{C}{\mathop{\times}}\mathbb{Z}\backslash \{(0,j):j\ne 0\}\overset{p_1}{\to}\mathbb{C}\,,
\end{align*}
where the embedding into ${\mathbb{C}}{\mathop{\times}}\mathbb{C}^*{\mathop{\times}}\mathbb{Z}$ is given by $(x,j)\mapsto (x,x^{-j},j)\,,$ for $x\ne 0\,,$ and $(0,0)\mapsto (0,1,0)\,.$
We can then form the quotient to get another module, denoted $\mathbb{C}^*_{\mathbb{C}}(*\{0\})\,.$ Formally, we have the following:
\theoremstyle{definition}\begin{definition}
 We define the space $\mathbb{C}^*_{\mathbb{C}}(*\{0\})$ as
\begin{align*}
   \mathbb{C}^*_{\mathbb{C}}(*\{0\}):= {\mathbb{C}}{\mathop{\times}}\mathbb{C}^*{\mathop{\times}}\mathbb{Z}/\sim\,,\,(x,y,i)\sim (x,x^{-j}y,i+j)\,,\;x\ne 0\,.
\end{align*}
$\blacksquare$\end{definition}
\begin{proposition}The space $\mathbb{C}^*_{\mathbb{C}}(*\{0\})$ is a complex manifold and there is a holomorphic surjective submersion $\pi:M\to\mathbb{C}$ given by $\pi(x,y,i)=x\,$ The space $\mathbb{C}^*_{\mathbb{C}}(*\{0\})$ is a family of abelian groups with product defined by a
\begin{align*}
(x,y,i)\cdot(x,y',j)=(x,yy',i+j)\,.
\end{align*}
It is a $\mathbb{C}^*\ltimes\mathbb{C}$-module with action given by
\begin{align*}
    (a,x)\cdot(x,y,i)=(ax,a^{-i}y,i)\,,
\end{align*}
and there is a short exact sequence of modules given by
\begin{align*}
    0\to \mathbb{C}{\mathop{\times}}\mathbb{Z}\backslash \{(0,j):j\ne 0\}\to {\mathbb{C}}{\mathop{\times}}\mathbb{C}^*{\mathop{\times}}\mathbb{Z}\to \mathbb{C}^*_{\mathbb{C}}(*\{0\})\to 0\,. 
\end{align*}
The fiber of $\mathbb{C}^*_{\mathbb{C}}(*\{0\})$ over a point $x\ne 0$ is isomorphic to $\mathbb{C}^*\,,$ and the fiber over $x=0$ is isomorphic to $\mathbb{C}^*{\mathop{\times}}\mathbb{Z}\,.$ 
\end{proposition}
\begin{proof}
We prove that it is a complex manifold. First we show that we can cover the space with charts whose transition functions are holomorphic. For each $i\in\mathbb{Z}\,,$ we get a chart given by $\mathbb{C}\times\mathbb{C}^*\,,$ taking $(x,y,i)\mapsto (x,y)\,.$ On the intersection between the $i$ and $j$ coordinate systems, the transition function is given by $(x,y)\mapsto (x,x^{-j}y)\,,$ which is holomorphic. 
\newline\newline To prove the space is Hausdorff, we observe that away from $x=0\,,$ the space is just $\mathbb{C}^*\times\mathbb{C}^*\,.$ Now take two points $(0,y,i)\,,(x,y',j)\,,$ $x\ne 0\,.$ We get disjoint neighborhoods of these points by choosing small enough neighborhoods $U_i\,, U_j\,,$ such that the projections onto the $x$-coordinate are disjoint. Now given two distinct points $(0,y,i)\,,(0,y',j)\,,$ with $j>i$ we obtain two disjoint neighborhoods by choosing $x\in \mathbb{C}$ such that $|x^{i-j}y|>|y'|\,,$ and then choosing small enough disks around $y\,,y'\,.$ Now suppose we take two distinct points $(0,y,i)\,,(0,y',i)\,.$ We get two disjoint neighborhoods by choosing disjoint neighborhoods of $y\,,y'\in\mathbb{C}^*\,,$ and taking all $x\in\mathbb{C}^*\,.$
\end{proof}
\begin{proposition}\label{sheaf identification}The sheaf $\mathcal{O}(\mathbb{C}^*_{\mathbb{C}}(*\{0\}))$ (where sections here are taken to be holomorphic) is isomorphic to the sheaf of meromorphic functions on $\mathbb{C}$ with poles or zeroes only at $x=0\,,$ denoted $\mathcal{O}^*(*\{0\})\,.$
\end{proposition}
\begin{proof}
Consider the morphism of sheaves defined as follows: for an open set $U\subset\mathbb{C}$ and a holomorphic section $s(x)=(x,f(x),i)$ of $\mathbb{C}^*_{\mathbb{C}}(*\{0\})$ over $U\,,$ define a meromorphic function on $U\,,$ with a possible pole/zero only at $x=0\,,$ by $x^if(x)\,,\,x\in U\,.$ This map is an isomorphism of sheaves.
\end{proof}
Now to any $G$-module there is an associated $G$-representation, and the representation associated to $\mathbb{C}^*_{\mathbb{C}}(*\{0\})$ is the trivial one, ie. $\mathfrak{m}\cong\mathbb{C}{\mathop{\times}}\mathbb{C}$ with the projection map being the projection onto the first factor, and the action of $\mathbb{C}^*\ltimes\mathbb{C}$ is given by 
\begin{align*}
(a,x)\cdot(x,y)=(ax,y)\,.
\end{align*}
We identify $\mathfrak{m}$ with points $(x,y,0)\in\mathbb{C}{\mathop{\times}}\mathbb{C}{\mathop{\times}}\mathbb{Z}\,,$ where the second $\mathbb{C}$ is identified with the Lie algebra of $\mathbb{C}^*\,.$ The sheaf of sections of $\mathfrak{m}$ is naturally isomorphic to the sheaf of $\mathbb{C}$-valued functions on $\mathbb{C}\,.$
\begin{proposition}The Chevalley-Eilenberg complex associated to $\mathbb{C}^*_{\mathbb{C}}(*\{0\})$ is isomorphic to the complex
\begin{align*}
    \mathcal{O}_\mathbb{C}^*(*\{0\})\overset{\mathrm{dlog}}{\to}\Omega^1_\mathbb{C}(\log D)\,.
\end{align*}
\end{proposition}
\begin{proof}We will compute $\mathrm{d}_\text{CE}\,\mathrm{log}:$ consider the meromorphic function $x^nf(x)\,,$ $x\in U\,,$ where $f$ is holomorphic and nonvanishing. We identify it with the local section of $\mathbb{C}^*_{\mathbb{C}}(*\{0\})$ given by $s(x)=(x,f(x),n)\,.$ Now the anchor map is given by \begin{align*}
    \alpha:\textit{Lie}(\mathbb{C}^*\ltimes\mathbb{C})\to T\mathbb{C}\,,\,\alpha(\partial_x,x)=x\partial_x\,.
    \end{align*}
Then we can compute that
\begin{align*}
&\tilde{L}_{(\partial_x,x)}s(x)=\frac{d}{d\varepsilon}\Big\vert_{\varepsilon=0}\,(x,e^{n\varepsilon}f(e^{\varepsilon}x)f(x)^{-1},0)
=(x,n+xf'(x)f(x)^{-1},0)
\\&=(x,\mathrm{dlog} (x^nf)\,(x\partial_x),0)=(x,\mathrm{dlog} (x^nf)\,\alpha(\partial_x,x),0)\,,
\end{align*}
so $f$ differentiates to $\mathrm{dlog} (x^nf)\,,$ so that $\mathrm{d}_\text{CE}\,\mathrm{log}$ corresponds to $\mathrm{dlog}$ under the identification of sheaves used in Proposition~\ref{sheaf identification}. This completes the proof.
\end{proof}
\subsection{The Module $\mathbb{C}^*_X(*D)$ and Pair$(X,D)$}
Here we will generalize the construction in the previous section to arbitrary complex manifolds and smooth divisors.
\begin{proposition}
Let $X$ be a complex manifold of complex dimension $n\,,$ and let $D$ be a smooth divisor. Then there is a canonical family of abelian groups $\mathbb{C}^*_X(*D)\to X$ such that $\mathcal{O}(\mathbb{C}^*_X(*D))$ (where sections here are taken to be holomorphic) is isomorphic to $\mathcal{O}^*(*D)\,,$ the sheaf of nonvanishing meromorphic  functions with poles or zeros only on $D\,.$
\end{proposition} 
\begin{proof}We can construct a family of abelian groups as follows: choose an open cover $\{\mathbb{D}^n_i\}_i$ of $X$ by polydiscs $($ie.  $\mathbb{D}_i=\{z\in\mathbb{C}:|z|<1\})\,,$ with coordinates $(x_{i,1},\mathbf{x}_i)=(x_{i,1}\,,x_{i,2}\,,\ldots\,,x_{i,n})$ on $\mathbb{D}^n_i\,,$ in such a way that 
\begin{align*}D\cap \mathbb{D}^n_i=\{x_{i,1}=0\}\,.
\end{align*}
Then on $\mathbb{D}^n_i$ form the family of abelian groups $\mathbb{D}^n_i{\mathop{\times}}\mathbb{C}^*{\mathop{\times}}\mathbb{Z}/\sim\,,$ where \begin{align*}
(x_{i,1},\mathbf{x}_i,y,k)\sim (x_{i,1},\mathbf{x}_i,x_{i,1}^{-l}y,k+l) \textit{  for } x_{i,1}\ne 0\,,
\end{align*}
where the surjective submersion is given by the projection onto $(x_{i,1},\mathbf{x}_i)\,,$ and where the product is given by 
\begin{align*}
    (x_{i,1},\mathbf{x}_i,y,k)\cdot (x_{i,1},\mathbf{x}_i,y',l)=(\mathbf{x}_i,yy',k+l)\,.
\end{align*}
We can glue these families of abelian groups together in the following way: on $\mathbb{D}^n_i\cap\mathbb{D}^n_j$ we have a nonvanishing holomorphic function $g_{ij}$ satisfying $x_{j,1}=g_{ij}x_{i,1}\,.$ Now let 
\begin{align*}
(x_{i,1},\mathbf{x}_i,y,k)\sim(x_{j,1},\mathbf{x}_j,g_{ij}^{-k}y,k)\,.
\end{align*}
This gluing preserves the fiberwise group structure, hence we obtain a family of abelian groups, denoted 
\begin{align*}
    \mathbb{C}^*_X(*D)\overset{\pi}{\to}X\,.
\end{align*}
As in the previous section, where this was done for $(X,D)=(\mathbb{C},\{0\})\,,$ the sheaf $\mathcal{O}(\mathbb{C}^*_X(*D))$ is isomorphic to $\mathcal{O}^*(*D)\,.$
\end{proof}
\begin{proposition}[see~\cite{pym}] There is a terminal integration of $T_X(-\log{} D) (denoted by denoted $\text{Pair}(X,D))\,,$\,,$ the Lie algebroid whose sheaf of
sections is isomorphic to the sheaf of sections of $T_X$ which are tangent to $D\,.$ 
\end{proposition}
\begin{proof}The terminal integration, $\text{Pair}(X,D)\,,$ can be described locally as follows (here the notation is as in the previous proposition):
the set of morphisms $\mathbb{D}^n_i\to\mathbb{D}^n_j$ is given by all
\begin{align*}
  &(a,\mathbf{x}_j,x_{i,1},\mathbf{x}_i)\in\mathbb{C}^*{\mathop{\times}}\mathbb{D}_j^{n-1}{\mathop{\times}}\mathbb{D}_i{\mathop{\times}}\mathbb{D}_i^{n-1}
  \\&\text{such that } (ax_{i,1},\mathbf{x}_j)\in \mathbb{D}^n_j\,.
\end{align*}
The source, target and multiplication maps are:
\begin{align*}
    &s(a,\mathbf{x}_j,x_{i,1}\,,\mathbf{x}_i)=(x_{i,1}\,,\mathbf{x}_i)\in \mathbb{D}_i^{n}\,,
    \\&t(a,\mathbf{x}_j,x_{i,1}\,,\mathbf{x}_i)=(ax_{i,1},\mathbf{x}_j)\in\mathbb{D}_j^{n}\,,
    \\&(a',\mathbf{x}_k,ax_{i,1},\mathbf{x}_j)\cdot(a,\mathbf{x}_j,x_{i,1}\,,\mathbf{x}_i)
    \\&=(a'a,\mathbf{x}_k,x_{i,1}\,,\mathbf{x}_i)\in\mathbb{C}^*{\mathop{\times}}\mathbb{D}_k^{n-1}{\mathop{\times}}\mathbb{D}_i{\mathop{\times}}\mathbb{D}_i^{n-1}\,.
\end{align*}
The gluing maps on the groupoid are induced by the gluing maps on $X\,,$ that is, 
\begin{align*}
   &(a,\mathbf{x}_j,x_{i,1}\,,\mathbf{x}_i)\sim \Big(a\frac{g_{jl}(ax_{i,1})}{g_{ik}(x_{i,1})},\mathbf{x}_l,x_{k,1}\,,\mathbf{x}_k\Big) 
   \end{align*}
  if
\begin{align*}
&(x_{i,1},\mathbf{x}_i)\in \mathbb{D}_i^{n}\sim (x_{k,1},\mathbf{x}_k)\in \mathbb{D}_k^{n}\,, 
\\&(ax_{i,1},\mathbf{x}_j)\in \mathbb{D}_j^{n}\sim (x_{l,1},\mathbf{x}_l)\in \mathbb{D}_l^{n}\,.
\end{align*}
\end{proof}
\begin{proposition}The morphism 
\begin{align}\label{TX-module}
\mathrm{dlog}:\mathcal{O}^*(*D)\to \Omega_X^1(\log{} D)
\end{align}
endows $\mathbb{C}^*_X(*D)$ with the structure of a $T_X(-\log D)$-module, and this structure integrates to give $\mathbb{C}^*_X(*D)$ the structure of a $\mathrm{Pair}(X,D)$-module.
\end{proposition}
\begin{proof}Define an action of $\mathrm{Pair}(X,D)$ on $\mathbb{C}^*_X(*D)$ as follows (the notation is as in the previous two propositions):
\begin{align*}
    &(a,\mathbf{x}_j,x_{i,1}\,,\mathbf{x}_i)\cdot(x_{i,1},\mathbf{x}_i,y,k)
    \\&=(ax_{i,1},\mathbf{x}_j,a^{-k}y,k)\,.
\end{align*}
This is a well-defined action by fiberwise isomorphisms, and it indeed differentiates to the $T_X(-\log{}D)$-module defined by
\eqref{TX-module}.
\end{proof}
Essentially the same construction can be done in the case that $D$ is a simple normal crossing divisor. In a neighborhood $U$ of a simple crossing divisor which is biholomorphic to a polydisk, we can choose coordinates $\mathbf{x}=(x_1,\ldots,x_n)$ on $\mathbb{D}^n$ such that
the simple normal crossing divisor is given by $x_1\cdots x_k=0\,.$ Then
\begin{align*}
   & \mathbb{C}^*_X(*D)\vert_{\mathbb{D}^n}=\mathbb{D}^n{\mathop{\times}}\mathbb{C}^*{\mathop{\times}}\mathbb{Z}^k/\sim\,,
    \,(\mathbf{x},y,\mathbf{i})\sim (\mathbf{x},x_{j_1}^{-m_{j_1}}\cdots x_{j_l}^{-m_{j_l}}y,\mathbf{i}+\mathbf{m})
   \\&\text{away from } x_{j_1}\cdots x_{j_l}=0\,, \text{ where } j_1,\ldots,j_l\in \{1,\ldots,k\}
   \\&\text{and where } m_{j_1},\ldots, m_{j_l}\text{  are the nonzero components of }\mathbf{m}\in\mathbb{Z}^k\,.
\end{align*}
Alternatively, it can locally be described as \begin{align*}
    \mathbb{C}^*_X(*{\{x_1=0\}})\otimes_{\mathbb{C}^*}\cdots\otimes_{\mathbb{C}^*}\mathbb{C}^*_X(*{\{x_k=0\}})\,,
    \end{align*}
where the $\mathbb{C}^*$-action is the one induced by the action of $\mathbb{C}^*_U$ on $\mathbb{C}^*_X(*D)\,,$ which comes from the embedding $\mathbb{C}^*_U\xhookrightarrow{}\mathbb{C}^*_X(*D)\,.$
\newline\newline To summarize this section, we have proven the following:
\begin{theorem}
Let $X$ be a complex manifold and let $D$ be a simple normal crossing divisor. There is a family of abelian groups 
\begin{align*}
    \mathbb{C}^*_X(*D)\overset{\pi}{\to}X
\end{align*}
whose sheaf of holomorphic sections is isomorphic to $\mathcal{O}_X^*(*D)\,.$
Furthermore, there is a canonical action of $\mathrm{Pair}(X,D)$ on $M$ making it into a $\mathrm{Pair}(X,D)$-module, and this module structure integrates the canonical $T_X(-\log{}D)$-module structure on $M$
induced by the morphism
\begin{align*}
\mathrm{dlog}:\mathcal{O}_X^*(*D)\to \Omega_X^1(\log{} D)\,.
\end{align*}
\end{theorem}

\section{Integration by Prequantization}
In this section we describe an alternative approach to integration of classes in Lie algebroid cohomology that may sometimes be used, and which doesn't directly involve the van Est map (more accurately, this method could be combined with the previous method). We call it integration by prequantization because in the case that the Lie algebroid is the tangent bundle and one is trying to integrate a $2$-form $\omega$, this method uses the line bundle whose first Chern class is the cohomology class of $\omega\,.$ 
We will first describe this method and then give some examples. 
\begin{comment}First we define Poincar\'{e} modules.
\begin{definition}
Let $G\rightrightarrows G^0$ be a Lie groupoid. A Poincar\'{e} $G$-module is a $G$-module $M$ such that the associated Chevalley-Eilenberg complex satisfies the Poincar\'{e} lemma, that is the Chevalley-Eilenberg complex is locally exact.
$\blacksquare$\end{definition}
Now let $M$ be a Poincar\'{e} $G$-module, then 
\begin{align*}
H^*(G^0,\ker[\mathcal{O}(M)\overset{\mathrm{d}_{CE}\mathrm{log}}{\to}\mathcal{C}^1(\mathfrak{g},M)]\,)\cong H^*(\mathfrak{g},M)\,.
\end{align*}
\end{comment}
\\\\Suppose we have a $G$-module $N$ and we are interested in integrating a class in the cohomology of the truncated complex, $\alpha\in H^*_0(\mathfrak{g},N)\,.$ 
Now suppose we have a $G$-module $M$ such that $\mathfrak{m}=\mathfrak{n}\,,$ and such that there is a map $N\to M$ of $G$-modules which differentiates to the identity map on $\mathfrak{n}\,.$ 
In this case
the morphism 
\[
\mathcal{O}(M)\xrightarrow{\text{d}_{\text{CE}}\text{log}}\mathcal{C}^1(\mathfrak{g},M)\cong \mathcal{C}^1(\mathfrak{g},N)
\]
induces a morphism
\begin{align*}
H^*(G^0,\mathcal{O}(M))\to H_0^*(\mathfrak{g},N)\,.    
\end{align*}
Then one can try lift $\alpha$ to a class $\tilde{\alpha}\in H^*(G^0,\mathcal{O}(M))\,.$ If a lift can be found, then one can attempt to integrate $\alpha$ to a class in $H^*_0(G,N)$ by showing that $\delta^*\tilde{\alpha}$ is in the image of the map $H^*_0(G,N)\to H^*_0(G,M)\,.$ If this succeeds then this class in $H^*_0(G,N)$ integrates $\alpha\,.$ We can summarize this method with the following proposition:
\begin{proposition} Let $G$ be a Lie groupoid, and let $N,M$ be $G$-modules with the same underlying Lie algebroids $\mathfrak{n}\,.$ Suppose further that there is a map of $G$-modules $f:N\to M$ which differentiates to the identity on $\mathfrak{n}$ (in particular this means that the Lie algebroids of $N$ and $M$ are the same as $G$-representations). The following diagram is commutative:
\[
    \begin{tikzcd}[row sep=3em, column sep = 3em]
 & H^*_0(G,M) \arrow[bend left=60,"VE_0"]{dd}
 \\ H^*_0(G,N)\arrow{ur}{f}\arrow{d}{VE_0}& H^*(G^0,\mathcal{O}(M))\arrow{u}{\delta^*}\arrow{d}{\mathrm{d}_{\mathrm{CE}}\mathrm{log}} 
 \\H^*_0(\mathfrak{g},N)\arrow{r}{\cong}& H^*_0(\mathfrak{g},M)
\end{tikzcd}
\]
\end{proposition}
\theoremstyle{definition}\begin{exmp}
Let $X$ be a manifold and let $\omega$ be a closed $2$-form which has integral periods. Then there is a class $g\in H^1(X,\mathcal{O}^*)$ which lifts $\omega\,,$ ie. a principal $\mathbb{C}^*$-bundle. We then have that $\delta^*g\in H^1_0(\textrm{Pair}(X),\mathbb{C}^*_X)$ integrates $\omega\,.$ 
\end{exmp}
\theoremstyle{definition}\begin{exmp}
Consider the trivial $(\mathbb{C}^*\ltimes\mathbb{C}\rightrightarrows \mathbb{C})$-module $\mathbb{C}^*_{\mathbb{C}}\,,$ and let $\mathfrak{g}$ be its Lie algebroid. Consider the class in $H^0_0(\mathfrak{g},\mathbb{C}^*_{\mathbb{C}})$ given by $\frac{dz}{z}\,.$ This class is not in the image of 
\[
\mathcal{O}^*_{\mathbb{C}}\xrightarrow{\mathrm{d}\mathrm{log}}\mathcal{C}^1(\mathfrak{g},\mathbb{C}^*_{\mathbb{C}})\,.
\]
However, $\mathbb{C}^*_{\mathbb{C}}\xhookrightarrow{} \mathbb{C}^*_{\mathbb{C}}(*\{0\})$ $($where $\mathbb{C}^*_{\mathbb{C}}(*\{0\})$ is as in the previous section$)\,,$ and they have the same Lie algebroids, and in addition the class $\frac{dz}{z}$ is in the image of 
\[
\mathcal{O}^*_{\mathbb{C}}(*D)\xrightarrow{\mathrm{d}_{CE}\mathrm{log}}\mathcal{C}^1(\mathfrak{g},\mathbb{C}^*_{\mathbb{C}})\,,
\]
namely $\mathrm{d}\mathrm{log}\,z=\frac{dz}{z}\,.$ We then have that $\delta^*z\,(a,z)=a\,,$ which is $\mathbb{C}^*$-valued. Hence the morphism $(a,z)\mapsto a$ integrates $\frac{dz}{z}\,.$
\end{exmp}
To get examples involving the integration of extensions, we have the following proposition:
\begin{proposition}
Let $X$ be a complex manifold with smooth divisor $D\,,$ and let $\Pi_1(X,D)\rightrightarrows X$ be the source simply connected integration of $T_X(-\log{D})\,.$ Then the subgroup of classes in $H^1_0(T_X(-\log{D}),\mathbb{C}_{X})$ which are integral on $X\backslash D$ embeds into $H^1_0(\Pi_1(X,D),\mathbb{C}^*_{X})\,.$
\end{proposition}
\begin{proof}
Let $\omega\in H^1_0(T_X(-\log{D}),\mathbb{C}_X)$ be a class which is prequantizable, which means that it is in the image of the map 
\[
H^1(X,\mathcal{O}^*_X(*D))\to H^1_0(T_X(-\log{D}),\mathbb{C}_X)\,.
\]
It is proved in~\cite{luk} that this is equivalent to $\omega$ being integral on $X\backslash D\,.$ 
\\\\There is a short exact sequence of $\Pi_1(X,D)$-modules
\begin{align*}
    0\to \mathbb{C}^*_X\overset{\iota}{\to} \mathbb{C}^*_X(*D)\overset{\pi}{\to} \mathrm{\acute{e}t}(\iota_*\mathcal{O}(\mathbb{Z}_D))\to 0
\end{align*}
$($where $\iota:D\to X$ is the inclusion and \'{e}t means the \'{e}tal\'{e} space, which may be non-Hausdorff, but this is fine $)\,.$ From this we get the long exact sequence 
\begin{align*}
&H^0_0(\Pi_1(X,D),\mathrm{\acute{e}t}(\iota_*\mathcal{O}(\mathbb{Z}_D)))\to H^1_0(\Pi_1(X,D),\mathbb{C}^*_X)\to H^1_0(\Pi_1(X,D),\mathbb{C}^*_X(*D))
\\&\to H^1_0(\Pi_1(X,D),\mathrm{\acute{e}t}(\iota_*\mathcal{O}(\mathbb{Z}_D)))\,.    
\end{align*}
Now $H^0_0(\Pi_1(X,D),\mathrm{\acute{e}t}(\iota_*\mathcal{O}(\mathbb{Z}_D)))=0$ since a morphism of groupoids must be $0$ on the identity bisection, so since the fibers of $\mathrm{\acute{e}t}(\iota_*\mathcal{O}(\mathbb{Z}_D))$ are discrete and the source fibers $\Pi_1(X,D)$ are connected, any such morphism must be identically $0\,.$ So we get the long exact sequence
\begin{align*}
  0  \to H^1_0(\Pi_1(X,D),\mathbb{C}^*_X)\to H^1_0(\Pi_1(X,D),\mathbb{C}^*_X(*D))\to H^1_0(\Pi_1(X,D),\mathrm{\acute{e}t}(\iota_*\mathcal{O}(\mathbb{Z}_D)))\,.
\end{align*}
If we let $\alpha\in H^1(X,\mathbb{C}_X^*(*D))\,,$ then $t^*\alpha-s^*\alpha\in H^1_0(\Pi_1(X,D),\mathbb{C}_X^*(*D))\,,$ and 
\begin{align*}
\pi(t^*\alpha-s^*\alpha)=t^*\pi(\alpha)-s^*\pi(\alpha)=0\,,
\end{align*}
where the latter equality follows from the fact that $\pi(\alpha)$ is a module for the full subgroupoid over $D\,,$ which follows from the following: there is a morphism from the full subgroupoid over $D$ to $\Pi_1(D)\,,$ and $\pi(\alpha)$ is a module for $\Pi_1(D)$ since $\pi(\alpha)$ is a local system. 
\\\\Hence there is a unique lift of $\alpha$ to $H^1_0(\Pi_1(X,D),\mathbb{C}^*_X)\,.$ Hence all of the prequantizable classes in $H^2(T_X(-\log{D}),\mathbb{C}_X)$ integrate to classes in $H^1_0(\Pi_1(X,D),\mathbb{C}^*_X)\,,$
\end{proof}
What this proposition means is that any closed logarithmic $2$-form on a complex manifold $X$ with smooth divisor $D\,,$ which has integral periods on $X\backslash D\,,$ defines a $\mathbb{C}^*$-groupoid extension of $\Pi_1(X,D)\,.$
\theoremstyle{definition}\begin{exmp}We can specialize to the case $X=\mathbb{P}^2$ and where $D$ is a smooth projective curve of degree $\ge 3$ and genus $g$ in $\mathbb{P}^2\,.$ Then as proved in~\cite{luk}, the prequantizable subgroup of $H^1_0(T_{\mathbb{P}^2}(-\log{D}),\mathbb{C}^*_X)$ is isomorphic to $\mathbb{Z}^{2g}\,.$ Hence $\mathbb{Z}^{2g}\xhookrightarrow{}H^1_0(\Pi_1(\mathbb{P}^2,D),\mathbb{C}^*_{\mathbb{P}^2})\,.$
\end{exmp}
\appendix
\section{Appendix}
\subsection{Derived Functor Properties}\label{derived functor}
In this section we discuss some vanishing results for the derived functors of locally fibered maps. These results are particularly useful
when using the Leray spectral sequence.
\begin{definition}
A map $f:X\to Y$ between topological spaces is called locally fibered if for all points $x\in X$ there exists open sets $U\ni x$ and $V\ni f(x)\,,$ a topological space $F$ and a homeomorphism $\phi:U\to F{\mathop{\times}} V$ such that the following diagram commutes:
\[
\begin{tikzcd}
U \arrow{r}{\phi}\arrow{dr}{f} & F{\mathop{\times}} V\arrow{d}{p_2} \\
& V
\end{tikzcd}
\]
$\blacksquare$\end{definition}
\theoremstyle{definition}\begin{exmp}
If $X\,,Y$ are smooth manifolds and $f:X\to Y$ is a surjective submersion, then $f$ is locally fibered.
\end{exmp}
\begin{proposition}[The Canonical Resolution]\label{canonical resolution}
Let $M\to X$ be a family of abelian groups. Then there is a canonical acyclic resolution of $\mathcal{O}(M)$ that differs from the Godement resolution. It is given by the following: for a sheaf $\mathcal{S}\,,$ let $\mathbb{G}^0(\mathcal{S})$ be the first sheaf in the Godement resolution of $\mathcal{S}\,,$ ie. the sheaf of germs of $\mathcal{S}\,.$ Let $\mathcal{G}^0(M)$ be the sheaf of all sections of $M$ (including discontinuous ones). Let 
\begin{align*}
    \mathcal{G}^{n+1}(M)=\mathbb{G}^0(\textrm{coker}[\mathcal{G}^{n-1}(M)\to \mathcal{G}^n(M)])
\end{align*}
for $n\ge 0\,,$ where $\mathcal{G}^{-1}(M):=\mathcal{O}(M)\,.$ We then have the following acylic resolution of $\mathcal{O}(M):$
\begin{align*}
    0\to\mathcal{O}(M)\to \mathcal{G}^{0}(M)\to \mathcal{G}^1(M)\to\cdots\,.
\end{align*}
\end{proposition}
\begin{definition}[see \cite{Bernstein}]\label{def n acyclic}
A continuous map $f:X\to Y$ is called $n$-acyclic if it satisfies the following conditions:
\begin{enumerate}
    \item For any sheaf $\mathcal{S}$ on $Y$ the adjunction morphism $\mathcal{S}\mapsto R^0f_*(f^{-1}\mathcal{S})$ is an isomorphism and 
    $R^if_*(f^{-1}\mathcal{S})=0$ for all $i=1\,,\ldots\,,n\,.$
    \item For any base change $\tilde{Y}\to Y$ the induced map $f:X{\mathop{\times}}_Y \tilde{Y}\to\tilde{Y}$ satisfies property 1.
\end{enumerate}
$\blacksquare$\end{definition}
\begin{theorem}[see \cite{Bernstein}, criterion 1.9.4]\label{n acyclic}
Let $f:X\to Y$ be a locally fibered map. Suppose that all fibers of $f$ are $n$-acyclic (ie. $n$-connected). Then $f$ is $n$-acyclic. 
\end{theorem} 
\begin{corollary}\label{restriction is zero}
Let $f:X\to Y$ be a locally fibered map and suppose that all fibers of $f$ are $n$-acyclic. Let $M$ be a family of abelian groups. Then if \begin{align*}
\alpha\in R^{n+1}f_*(f^{-1}\mathcal{O}(M))(Y)
\end{align*}
satisfies $\alpha\vert_{f^{-1}(y)}=0$  for all $y\in Y$ $($note that $\alpha\vert_{f^{-1}(y)}\in H^{n+1}(f^{-1}(y),f^{-1}M_y))\,,$ then $\alpha=0\,.$
\end{corollary}
\begin{proof}
By Proposition~\ref{canonical resolution} we have the following resolution of $\mathcal{O}(M)\,:$
\begin{align*}
    0\to\mathcal{O}(M)\to \mathcal{G}^{0}(M)\to \mathcal{G}^1(M)\to\cdots\,.
\end{align*}
Since $f^{-1}$ is an exact functor, we obtain the following resolution of $f^{-1}\mathcal{O}(M):$
\begin{align}\label{resolution}
    0\to f^{-1}\mathcal{O}(M)\to f^{-1}\mathcal{G}^{0}(M)\to f^{-1}\mathcal{G}^1(M)\to\cdots\,.
\end{align}
Hence, \begin{align*}
    R^\bullet f_*(f^{-1}\mathcal{O}(M))=R^\bullet f_*(f^{-1}\mathcal{G}^{0}(M)\to f^{-1}\mathcal{G}^1(M)\to\cdots)\,.
\end{align*}
One can show that $\alpha\vert_{f^{-1}(y)}=0$ for all $y\in Y\iff \alpha\mapsto 0$ under the map induced by $f^{-1}\mathcal{O}(M)\to f^{-1}\mathcal{G}^{0}(M)\,,$
hence we obtain the result by using Theorem~\ref{n acyclic}.
\end{proof}
\begin{theorem}
Let $f:X\to Y$ be a locally fibered map such that all its fibers are $n$-acyclic. Let $M\to Y$ be a family of abelian groups. 
Then the following is an exact sequence for $0\le k\le n+1:$
\begin{align*}
& 0\to H^k(Y,\mathcal{O}(M))\overset{f^{-1}}{\to} H^k(X, f^{-1}\mathcal{O}(M))
\\& \to
\prod_{y\in Y} H^k(f^{-1}(y),\mathcal{O}((f^*M)\vert_{f^{-1}(y)}))\,.
\end{align*}
In particular, $H^k(Y,\mathcal{O}(M))\overset{}{\to} H^k(X, f^{-1}\mathcal{O}(M))$
is an isomorphism for $k=0\,,\ldots\,,n$ and is injective for $k=n+1\,.$
\end{theorem}
\begin{proof}
This follows from the Leray spectral sequence (see section $0$) and Corollary~\ref{restriction is zero}.
\end{proof}
Similarly, we can generalize this result to simplicial manifolds:
\begin{theorem}\label{spectral theorem}
Let $f:X^\bullet\to Y^\bullet$ be a locally fibered morphism of simplicial topological spaces such that, in each degree, all of its fibers are $n$-acyclic. Let $M_\bullet\to Y^\bullet$ be a simplicial family of abelian groups. 
\begin{comment}Let $\mathcal{S}_\bullet$ be a sheaf on $Y^\bullet\,.$ Then the map
\begin{equation*}\label{isomorphism}
H^k(Y^\bullet,\mathcal{O}(M_\bullet))\xrightarrow{f^{-1}}H^k(X^\bullet, f^{-1}\mathcal{O}(M_\bullet))
\end{equation*}
is an isomorphism for $k=0\,,\ldots\,,n$ and is injective for $k=n+1\,.$ Its image in degree $n+1$ is all $\alpha\in H^{n+1}(X^\bullet, f^{-1}\mathcal{O}(M_\bullet))$ such that $r(\alpha)\vert_{f^{-1}(y)}=0\,,$ for all $y\in Y^0\,.$ 
\end{comment}
Then the following is an exact sequence for $0\le k\le n+1:$
\begin{align*}\label{Short exact}
& 0\xrightarrow{} H^k(Y^\bullet,\mathcal{O}(M_\bullet))\xrightarrow{f^{-1}}H^k(X^\bullet, f^{-1}\mathcal{O}(M_\bullet))
\\ &\xrightarrow{}
\prod_{y\in Y^0} H^k(f^{-1}(y),\mathcal{O}((f^*M_0)\vert_{f^{-1}(y)}))\,.
\end{align*}
In particular, $H^k(Y^\bullet,\mathcal{O}(M_\bullet))\overset{}{\to} H^k(X^\bullet, f^{-1}\mathcal{O}(M_\bullet))$
is an isomorphism for $k=0\,,\ldots\,,n$ and is injective for $k=n+1\,.$\footnote{See Section~\ref{derived functor} for more details.}
\end{theorem}

\subsection{Lie Groupoids}\label{BG functor}
In this section we briefly review some important concepts in the theory of Lie groupoids.
\begin{definition}
A groupoid is a category $G\rightrightarrows G^0$ for which the objects $G^0$ and morphisms $G$ are sets and for which every morphism is invertible. A Lie groupoid is a groupoid $G\rightrightarrows G^0$ such that $G^0\,, G$ are smooth manifolds\begin{footnote}{We allow for the possibility that the manifolds are not Hausdorff, but all structure maps should be locally fibered.}\end{footnote}, such that the source and target maps, denoted $s\,,t$ respectively, are submersions, and such that all structure maps are smooth, ie.
\begin{align*}
    & i:G^0\to G
    \\ & m:G\sideset{_s}{_{t}}{\mathop{\times}} G\to G
    \\& \text{inv}:G\to G
\end{align*}
are smooth (these maps are the identity, multiplication/composition and inversion, respectively). A morphism between Lie groupoids $G\to H$ is a smooth functor between them.
$\blacksquare$\end{definition}
\begin{definition}
Let $G\rightrightarrows G^0\,, K\rightrightarrows K^0$ be Lie groupoids. A Morita map $\phi:G\to K$ is a map such that
    \begin{enumerate}
        \item $\phi:G^0\to K^0$ is a surjective submersion
        \item The following diagram is Cartesian
        \[
\begin{tikzcd}
G \arrow{r}{(s,t)} \arrow{d}{\phi} & G^0{\mathop{\times}} G^0 \arrow{d}{(\phi\,,\phi)} \\
K \arrow{r}{(s,t)} & K^0{\mathop{\times}} K^0
\end{tikzcd}
\]
    \end{enumerate}
We say that $G\,,K$ are Morita equivalent groupoids if either:
\begin{enumerate}
    \item there is a Morita map between $G\to K$ or $K\to G$
    \item there is a third groupoid $H$ such that both $G\,,H$ and $H\,,K$ are Morita equivalent in the sense of 1.
\end{enumerate} 
Note that 1. is a special case of 2. In the case of 2. we say that $G\leftarrow H\rightarrow K$ is a Morita equivalence.
$\blacksquare$\end{definition}
\begin{definition}
There is a functor 
\[\mathbf{B}^\bullet:\text{groupoids}\to \text{simplicial spaces}\,,\,G\mapsto\mathbf{B}^\bullet G\,,
\]
where $\mathbf{B}^0 G=G^0\,,\,\mathbf{B}^1 G=G\,,$ and 
\[\mathbf{B}^n G=\underbrace{G\sideset{_t}{_{s}}{\mathop{\times}} G \sideset{_t}{_{s}}{\mathop{\times}} \cdots\sideset{_t}{_{s}}{\mathop{\times}} G}_{n \text{ times}}\,,
\]
the space of 
$n$-composable arrows. Here the face maps are the source and target maps for $n=1\,,$ and for $(g_0\,,\ldots\,,g_n)\in \mathbf{B}^{n+1}G\,,$
\begin{align*}
& d_{n+1,0}(g_0\,,\ldots\,,g_n)=(g_1\,,\ldots\,,g_n)\,, 
\\& d_{n+1,i}(g_0\,,\ldots\,,g_n)=(g_0\,,\ldots\,,g_{i-1}g_i\,,\hat{g}_i\,,\ldots\,,g_n)\,,\,\;1\le i\le n
\\& d_{n+1,n+1}(g_0\,,\ldots\,,g_n)=(g_0\,,\ldots\,,g_{n-1})\,.
\end{align*}
The degeneracy maps are $\text{Id}:G^0\to G$ for $n=0\,,$ and 
\begin{align*}
\\& \sigma_{n-1,i}(g_0\,,\ldots\,,g_{n-1})=(g_0\,,\ldots\,,g_{i-1}\,,\text{Id}(t(g_i))\,,\hat{g}_i\,,\ldots\,,g_{n-1})\,,\,\;0\le i\le n-1
\\& \sigma_{n-1,n}(g_0\,,\ldots\,,g_{n-1})=(g_0\,,\ldots\,,g_{i-1}\,,\hat{g}_i\,,\ldots\,,g_{n-1}\,,\text{Id}(s(g_{n-1})))\,.
\end{align*}
A morphism $f:G\to H$ gets sent to $\mathbf{B}^\bullet f:\mathbf{B}^\bullet G\to \mathbf{B}^\bullet H\,,$ which acts as $f$ does for $n=0\,,1\,,$ and
\begin{align*}
    \mathbf{B}^n f(g_0\,,\ldots\,,g_{n-1})=(f(g_0)\,,\ldots\,,f(g_{n-1}))
\end{align*}
for $n> 1\,.$
$\blacksquare$\end{definition}
\begin{comment}
\begin{lemma}
If $G\leftarrow Z\rightarrow A$ is a Morita equivalence of topological groupoids, then their is a canonical equivalence of
their categories of modules, which we also denote by $Z\,.$ That is, there is a functor $Z:Mod(G)\to Mod(A)$ which is an equivalence of categories.
\end{lemma}
\begin{lemma}
If $G\leftarrow Z\rightarrow A$ is a Morita equivalence of topological groupoids and $M$ is a $G$-module, then
\begin{align*}
    H^*(G,M)=H^*(A,Z(M))\,.
\end{align*}
\end{lemma}
\end{comment}
\subsection{Cohomology of Sheaves on Stacks}\label{appendix:Cohomology of Sheaves on Stacks}
In this section we briefly review the Grothendieck topology and sheaves on a differentiable stack, as well as their cohomology. The following definitions are based on~\cite{Kai}.
\begin{definition}
We call a family of morphisms $\{P_i\to P\}_i$ in $[G^0/G]$ a covering family if the corresponding family of morphisms on the base manifolds $\{M_i\to M\}_i$ is a covering family for the site of smooth manifolds, ie. a family of \'{e}tale maps such that $\coprod_i M_i\to M$ is surjective. This defines a Grothendieck topology on $[G^0/G]\,,$ thus we can now speak of sheaves on $[G^0/G]\,,$ ie. contravariant functors $\mathcal{S}:[G^0/G]\to\mathbf{Ab}$ such that the following diagram is an equalizer for all covering families $\{P_i\to P\}_i:$
\begin{align*}
\mathcal{S}(P)\to\prod_i \mathcal{S}(P_i)\rightrightarrows \prod_{i,j} \mathcal{S}(P_i{\mathop{\times}}_P P_j)\,.
\end{align*}
A morphism between sheaves $\mathcal{S}$ and $\mathcal{F}$ is a natural transformation from $\mathcal{S}$ to $\mathcal{F}\,.$
$\blacksquare$\end{definition}
\begin{definition}\label{stack cohomology}
Let $\mathcal{S}$ be a sheaf on $[G^0/G]\,.$ Define the global sections functor $\Gamma:\textrm{Sh}([G^0/G])\to \mathbf{Ab}$ by
\begin{align*}
\Gamma([G^0/G],\mathcal{S}):=\text{Hom}_{\text{sh}([G^0/G])}(\mathbb{Z}\,,\,\mathcal{S})\,,
\end{align*}
where $\mathbb{Z}$ is the sheaf on $[G^0/G]$ which assigns to the object
\[
\begin{tikzcd}
P\arrow{r}{}\arrow{d} & G^0
\\M
\end{tikzcd}
\]
the abelian group $H^0(M,\mathbb{Z})\,.$
$\blacksquare$\end{definition}
\begin{definition}
The global sections functor $\Gamma:\textrm{Sh}([G^0/G])\to \mathbf{Ab}$ is left exact and the category of sheaves on $[G^0/G]$ has enough injectives, so we define $H^*([G^0/G],\mathcal{S}):=R^*\Gamma(\mathcal{S})\,.$
$\blacksquare$\end{definition}
%\begin{theorem}\label{Morita}
%If $\mathcal{S}$ is a sheaf on $[G^0/G]$ and %$\phi:[K^0/K]\to [G^0/G]$ is an equivalence of %categories, we have that
%$H^\bullet(G,\mathcal{S})\cong %H^\bullet(K,\phi^{-1}\mathcal{S})\,.$
%\end{theorem}
\begin{theorem}[see~\cite{Kai}]\label{stack groupoid cohomology}
Let $\mathcal{S}$ be a sheaf on $[G^0/G]\,.$ 
%and let $\mathcal{S}(\mathbf{B}^\bullet G)$ be the %associated sheaf on $\mathbf{B}^\bullet G\,.$ 
Then 
\begin{align*}
    H^*([G^0/G],\mathcal{S})\cong H^*(\mathbf{B}^\bullet G,\mathcal{S}(\mathbf{B}^\bullet G))\,.
\end{align*}
\end{theorem}
 \subsection{Abelian Extensions}\label{abelian extensions}
Here we review abelian extensions and central extensions of Lie groupoids and Lie Algebroids. 
\begin{definition}Let $M$ be a $G$-module for a Lie groupoid $G\rightrightarrows G^0$. A Lie groupoid extension of $G$ by $M$ is given by a Lie groupoid 
$E\rightrightarrows G^0$ and a sequence of morphisms
\begin{align*}
   1\to M\overset{\iota}{\to} E\overset{\pi}{\to} G\to 1\,,
\end{align*}
such that $\iota\,,\,\pi$ are the identity on $G^0\,;$ such that $\iota$ is an embedding and $\pi$ is a surjective submersion; such that if $m\in M\,,\,e\in E$ satisfy $s(m)=s(e)\,,$ then $e\iota(m)=\iota(\pi(e)\cdot m)e\,;$ in addition, we require that $E\to G$ be principal $M$-bundle with respect to the right action. If $M$ is a trivial $G$-module then $E$ will be called a central extension. If $A$ is an abelian Lie group then associated to it is a canonical trivial $G$-module given by $A_{G^0}\,,$ and by an $A$-central extension of $G$ we will mean an
extension of $G$ by the trivial $G$-module $A_{G^0}\,.$ Furthermore, there is a natural action of $M$ on $E\,,$ and we assume that with this action $E$ is a principal $M$-bundle.
$\blacksquare$\end{definition}
\begin{definition}Let $\mathfrak{m}$ be a $\mathfrak{g}$-representation for a Lie algebroid $\mathfrak{g}\to N$. A Lie algebroid
extension of $\mathfrak{g}$ by $\mathfrak{m}$ is given by a Lie algebroid 
$\mathfrak{e}\to N$ and an exact sequence of the form
\begin{align*}
    0\to \mathfrak{m}\overset{\iota}{\to}\mathfrak{e}\overset{\pi}{\to}\mathfrak{g}\to 0\,,
\end{align*}
such that $\iota\,,\,\pi$ are the identity on $N\,,$ and such that if $X\,,Y$ are local sections over an open set $U\subset N$ of
$\mathfrak{m}\,,\mathfrak{e}\,,$ respectively, 
then $\iota(L_{\pi(Y)}X)=[Y,\iota(X)]\,.$ If $\mathfrak{m}$ is a trivial $\mathfrak{g}$-module then $\mathfrak{e}$ will be called a central extension. Similarly to the previous definition, if $V$ is a finite dimensional vector space then associated to it is a canonical trivial $\mathfrak{g}$-module given by $N{\mathop{\times}} V\,,$ and by a $V$-central extension of $\mathfrak{g}$ we will mean an
extension of $\mathfrak{g}$ by the trivial $\mathfrak{g}$-module $N{\mathop{\times}} V\,.$
$\blacksquare$\end{definition}

\begin{proposition}[see \cite{Kai} and \cite{luk}]
With the above definitions, $H^1_0(G,M)$ classifies extensions of $G$ by $M\,,$ and $H^1_0(\mathfrak{g},M)$ classifies extensions of $\mathfrak{g}$ by $\mathfrak{m}\,.$
\end{proposition}


\begin{thebibliography}{9}
\bibitem{baez2}
Baez, John C., and Lauda, Aaron D.
\textit{Higher-dimensional algebra. V: 2-Groups.}
Theory and Applications of Categories [electronic only] 12 (2004): 423-491. http://eudml.org/doc/124217.
%\bibitem{baez}
%Baez, John C., and Crans, Alissa S.
%\textit{Higher-dimensional algebra. VI: Lie 2-algebras.}
%Theory and Applications of Categories [electronic only] 12 (2004): 492-538. http://eudml.org/doc/124264.
\bibitem{Kai}
Behrend, Kai and Xu, Ping. 
\textit{Differentiable Stacks and Gerbes.}
J. Symplectic Geom. 9 (2011), no. 3, 285-341.
\bibitem{Bernstein}
Bernstein, J. and Lunts, V.
\textit{Equivariant Sheaves and Functors.}
LNM 1578, (1991).
\bibitem{brylinski}
Brylinski, Jean-Luc.
\textit{Differentiable Cohomology of Gauge Groups.}
arXiv:math/0011069 [math.DG], (2000).
%\bibitem{mclaughlin}
%Brylinski, J.L. and Mclaughlin, D. A.
%\textit{The Geometry of Degree-Four Characteristic Classes and of Line Bundles on Loop Spaces I.}
%Duke Mathematical Journal, Vol. 75 (1994).
\bibitem{Crainic} 
Crainic Marius.
\textit{Differentiable and Algebroid Cohomology, van Est Isomorphisms, and Characteristic Classes.} 
Commentarii Mathematici Helvetici, Vol.78, (2003) pp. 681-721.
\bibitem{rui}
Crainic, Marius and Fernandes, Rui Loja.
\textit{Integrabiltiy of Lie Brackets.}
Annals of Mathematics, 157 (2003), 575-620.
%\bibitem{crainic2}
%Crainic, Marius. 
%\textit{Prequantization and Lie brackets.}
%J. Symplectic Geom. 2 (2004), no. 4, 579-602. 
\bibitem{zhu}
Crainic, Marius and Zhu, Chenchang.
\textit{Integrability of Jacobi and Poisson structures.}
Annales de l'Institut Fourier, Volume 57 (2007) no. 4, pp. 1181-1216.
%\bibitem{davis}
%Davis, James F. and Kirk, Paul.
%%\textit{Lecture Notes in Algebraic Topology.}
American Mathematical Society (2001).
\bibitem{Deligne} 
Deligne P. 
\textit{Th\'{e}orie de Hodge. III.} 
Inst. Hautes \'{E}tudes Sci. Publ. Math No. $\mathbf{44}$ (1974),
5-77.
%\bibitem{bo}
%Feng, Bo and Hanany, Amihay and He, Yang-Hui and Prezas, Nikolaos.
%\textit{Discrete Torsion, Non-Abelian Orbifolds and the Schur Multiplier.}
%Journal of High Energy Physics, 5 (2000).
\bibitem{gen kahler}
Gualtieri, Marco.
\textit{Generalized Kahler Geometry.}
arXiv:1007.3485 [math.DG], (2010).
\bibitem{pym}
Gualtieri, Marco and Li, Songhao and Pym, Brent.
\textit{The Stokes Groupoids.}
Journal für die reine und angewandte Mathematik (Crelles Journal), Vol. 2018 (2013).
\bibitem{luk}
Gualtieri, Marco and Luk, Kevin.
\textit{Logarithmic Picard Algebroids and Meromorphic Line Bundles.}
arXiv:1712.10125 [math.AG], (2017).
%\bibitem{eli}
%Hawkins, Eli.
%\textit{A Groupoid Approach to Quantization.}
%J. Symplectic Geom. 6 (2008), no. 1, 61-125.
\bibitem{andre}
Henriques, André.
\textit{Integrating L$\infty$-algebras.}
Compositio Mathematica (2008), Vol 144.
%\bibitem{greg}
%Karpilovsky, Gregory.
%\textit{The Schur Multiplier.}
%Oxford University Press (1987).
\bibitem{krepski}
Krepski, Derek.
\textit{Basic equivariant gerbes on non-simply connected compact simple Lie groups}
Journal of Geometry and Physics
Volume 133, November 2018, Pages 30-41.
\bibitem{Meinrenken} 
Li-Bland, David and Meinrenken, Eckhard.
\textit{On the van Est homomorphism for Lie groupoids.} 
L’Enseignement Mathématique,
Vol. 61, (2014).
\bibitem{erbe}
Meinrenken, Eckhard.
\textit{The Basic Gerbe Over a Compact Simple Lie Group.}
Enseign. Math. 49. (2002).
%\bibitem{murray}
%Murray, Michael K. and Roberts, David Michael and Stevenson, Danny and Vozzo, Raymond F.
%\textit{Equivariant Bundle Gerbes.}
%Advances in Theoretical and Mathematical Physics, Vol. 21 (2015).
\bibitem{nlab}
nlab authors.
\textit{nlab: Lie Group Cohomology} 
Revision 15. (2019).
https://ncatlab.org/nlab/show/Lie+group+cohomology.
\bibitem{Peters} 
Peters, Chris A.M. and Steenbrink, Joseph H.M.
\textit{Mixed Hodge Structures (A Series of Modern Surveys in Mathematics).} 
Springer (2008).
\bibitem{pym2}
Pym, Brent and Safronov, Pavel. 
\textit{Shifted Symplectic Lie Algebroids.}
International Mathematics Research Notices, rny215 (2018).
\bibitem{schommer}
Schommer-Pries, Christopher J.
\textit{Central Extensions of Smooth 2-Groups and a Finite-Dimensional String
2-Group.}
Geometry and Topology, Vol. 15 (2009).
\bibitem{Severa}
Severa, Pavol and Weinstein, Alan,
\textit{Poisson Geometry with a 3-Form Background.}
Progress of Theoretical Physics Supplement, Vol. 144, (2002), pp. 145-154.
\bibitem{Tu} 
Tu, Jean-Louis. 
\textit{Groupoid Cohomology and Extensions.} 
Transactions of the American Mathematical,
Vol. 358, No. 11, (2006), pp. 4721-4747
\bibitem{wockel}
Wagemann, Friedrich and Wockel, Christoph.
\textit{A Cocycle Model for Topological and Lie Group Cohomology.}
Trans. Amer. Math. Soc. 367 (2015), 1871-1909. 
\bibitem{konrad}
Waldorf, Konrad.
\textit{Multiplicative Bundle Gerbes With Connection.}
Differential Geometry and its Applications, Vol. 28, Issue 3 (2010), pp. 313-340.
\bibitem{weinstein}
Weinstein, Alan and Xu, Ping.
\textit{Extensions of symplectic groupoids and quantization.}
Journal für die reine und angewandte Mathematik. Vol. 417, (1991) pp. 159-190.
%\bibitem{Xu}
%Xu, Ping.
%\textit{Differentiable Stacks, Gerbes and Twisted K-Theory.}
%(2017).
\end{thebibliography}
\end{document}